\documentclass[12 pt]{amsart}

\usepackage{graphicx}
\usepackage{color}
\usepackage{amssymb, amsmath, amsthm}
\usepackage{mathptmx}

\usepackage{epsfig}
\usepackage{epstopdf}
\usepackage{graphicx}
\usepackage{pinlabel}

\newtheorem{thm}{Theorem}[section]

\newtheorem*{maincor}{Corollary \ref{cor:main} (Rephrased)}
\newtheorem{cor}[thm]{Corollary}

\newtheorem{defin}[thm]{Definition}
\newtheorem{lemma}[thm]{Lemma}

\newtheorem{rmk}[thm]{Remark}

\newcommand{\nil}{\varnothing}
\newcommand{\inter}[1]{\mathring{#1}}

\newcommand{\bdd}{\mbox{$\partial$}}

\newcommand{\boundary}{\partial}
\newcommand{\mc}[1]{\mathcal{#1}}
\newcommand{\cls}{\operatorname{cl}}

\begin{document}



\title[C-essential surfaces in $(M,T)$]{C-essential surfaces in (3-manifold, graph) pairs}

\author{Scott A. Taylor}

\author{Maggy Tomova}
   \thanks{The first author was partially supported by a grant from the Natural Sciences Division of Colby College. The second author was supported by a grant from the National Science Foundation.}
   
\begin{abstract}
Let $T$ be a graph in a compact, orientable 3--manifold $M$ and let $\Gamma$ be a subgraph. $T$ can be placed in bridge position with respect to a Heegaard surface $H$. We show that if $H$ is what we call $(T,\Gamma)$-c-weakly reducible in the complement of $T$ then either a ``degenerate'' situation occurs or $H$ can be untelescoped and consolidated into a collection of ``thick surfaces'' and ``thin surfaces''. The thin surfaces are c-essential (c-incompressible and essential) in the graph exterior and each thick surface is a strongly irreducible bridge surface in the complement of the thin surfaces. This strengthens and extends previous results of Hayashi-Shimokawa and Tomova to graphs in 3-manifolds that may have non-empty boundary. \end{abstract}
\maketitle

\section{Introduction}
Thin position has been an important tool in knot theory and 3-manifold topology. Gabai \cite{G} defined thin position for a knot in $S^3$ in his solution of the Poenaru Conjecture; it was also put to good use by Gordon and Luecke \cite{GL} in their solution of the knot complement problem. Thin position was extended to graphs in $S^3$ by Scharlemann and Thompson \cite{ST1} who used it to give a new proof of Waldhausen's classification of Heegaard splittings of $S^3$. Goda, Scharlemann, and Thompson \cite{GST} also used thin position for graphs in $S^3$ to prove that an unknotting tunnel for a tunnel number one knot in minimal bridge position can be slid and isotoped to lie in the bridge surface. They used this fact to give a new proof of the classification of unknotting tunnels for 2--bridge knots. Goda, Scharlemann, and Thompson's result is also foundational to Cho and McCullough's \cite{CM} recent important work on tunnel number one knots. 

In tandem with their development of thin position for knots and graphs in $S^3$, Scharlemann and Thompson \cite{ST1} described a very different thin position for 3--manifolds (based on earlier work of Casson and Gordon \cite{CG}). In this type of thin position, a 3--manifold is decomposed along incompressible surfaces (called ``thin surfaces'') into codimension 0 submanifolds each of which contains a strongly irreducible Heegaard surface (called a ``thick surface'').  A locally thin position for a 3-manifold can be obtained by ``untelescoping'' a Heegaard surface for the 3--manifold into thin and thick surfaces. In practice, a strongly irreducible Heegaard surface functions very much like an incompressible surface. Part of the power of thin position for a 3-manifold comes from the fact that an irreducible Heegaard surface in a 3-manifold is either strongly irreducible, or can be untelescoped into a collection of incompressible surfaces and strongly irreducible surfaces. In either case, we have surfaces that function like incompressible surfaces. This aspect of thin position for 3--manifolds has been used, for example, in work on the virtual Haken conjecture \cite{L, Ma}. The virtual Haken conjecture was recently proven by Agol, Groves, and Manning \cite{Ag}.

Hayashi and Shimokawa \cite{HS3} found a way of uniting (at least in spirit) the notions of thin position for a knot and thin position for a 3-manifold. Their version of thin position involves placing a properly embedded 1-manifold in a 3-manifold into bridge position with respect to a Heegaard surface for the 3-manifold and then untelescoping the Heegaard surface using compressing disks and bridge disks. Coward \cite{C} used Hayashi and Shimokawa's thin position to find an algorithm for determining the bridge number of a hyperbolic knot in $S^3$.

Tomova \cite{T1} strengthened Hayashi and Shimokawa's untelescoping operation by weakening the hypotheses on the sort of disks used in the untelescoping operation: she used so-called ``c-disks''. Tomova \cite{T2} used this new version of thin position in her work on the relationship between multiple bridge surfaces for a given knot, producing a version of the ``alternate Heegaard genus bounds distance'' theorem of Scharlemann and Tomova \cite{STo1}. Scharlemann and Tomova \cite{STo3} also use this version of thin position to show that 2--bridge knots have essentially unique bridge surfaces. Tomova's version of thin position, however, requires that the 3-manifold be boundary-less.

In the present work, we show that (except in a few degenerate situations) a bridge surface for a graph in a 3-manifold, possibly with boundary, can be transformed using untelescoping-like operations into a type of thin position for the graph where the thin surfaces are essential in the graph complement and the thick surfaces are strongly irreducible bridge surfaces in the complement of the thin surfaces. The union of thin and thick surfaces is called a ``multiple Heegaard splitting'' for the graph in the 3-manifold. 

In the thin position of Hayashi-Shimokawa and Tomova the proof that thin surfaces are not parallel to a component of the link relies on the difficult classification of certain bridge surfaces \cite{HS1, HS2}. So in our situation, proving that thin surfaces are essential relies on the classification of bridge surfaces for certain graphs in compressionbodies \cite{TT1}. (See Theorem \ref{thm:HSofCompBodies}.)

Since in the past it has been helpful to be able to untelescope a bridge surface for a knot along c-disks (and not just compressing disks), we prove a relative version of the theorem. The set up is as follows: We have a compact orientable 3-manifold $M$, a Heegaard surface $H \subset M$ and a graph $T \subset M$ in bridge position with respect to $H$. Inside $T$ we specify a subgraph $\Gamma$ that contains no vertices of $T$, and when we untelescope, we will use compressing disks for $H - T$ and so-called ``cut disks'' that intersect $\Gamma$. If $M$ is a closed 3-manifold and $T$ is a link then if $\Gamma$ is chosen to equal $T$, we obtain Tomova's thin position, although our conclusion that the thin surfaces are c-essential is stronger than her conclusion. If $M$ is a 3-manifold and $T$ is a properly embedded 1-manifold, then choosing $\Gamma = \nil$ gives Hayashi and Shimokawa's thin position. We note that, even in this case, if $\boundary M \neq \nil$ our conclusion that the thin surfaces are essential (i.e. incompressible and not boundary parallel) in the exterior of the graph is stronger than Hayashi and Shimokawa's conclusion that the thin surfaces are incompressible in the complement of the graph and not ``$T$-parallel'' to a component of $\boundary M$. 

We prove

\begin{maincor}
Let $M$ be a compact, orientable $3$-manifold containing a properly embedded graph $T$ and let $\Gamma$ be a subgraph of $T$ disjoint from the vertices of $T - \boundary T$. Furthermore, assume that $M-T$ is irreducible and that no sphere in $M$ intersects $T$ exactly once. Suppose $K$ is a bridge surface for $(M, T,\Gamma)$ such that no circle or edge component of $\Gamma$ is isotopically core with respect to $K$.  Then one of the following holds:
\begin{itemize}
\item there is a multiple $\Gamma$-Heegaard splitting $\mc{H}$ for $(M,T, \Gamma)$ so that each thin surface is $T$--essential and each thick surface is $(T,\Gamma)$-c-strongly irreducible in the component of $M - \mc{H}^-$ containing it. Furthermore, $\mc{H}$ is obtained by applying ``untelescoping-like'' operations to $K$.
\item $K$ contains a generalized stabilization.
\item $K$ is perturbed.
\item $K$ has a removable path.
\end{itemize}
\end{maincor}
All of the terms are defined in the next section.

In a forthcoming paper we will demonstrate how this result can be used to show that certain edges in a Heegaard spine can be made level with respect to a bridge surface.

\section{Definitions}

\subsection{Surfaces in $(M,T)$}

Let $T$ be a finite graph. Unless otherwise specified we assume that $T$ has no valence 2 vertices as such vertices can generally be deleted and their adjacent edges amalgamated. $T$ may contain closed loops without vertices. We say that $T$ is {\em properly embedded} in a $3$--manifold $M$ if $T \cap \bdd M$ is the set of all valence $1$ vertices of $T$. We will denote the pair $(M,T)$.

Suppose that $F \subset M$ is a surface such that $\bdd F \subset (\bdd M \cup T)$. Then $F$ is {\em  $T$--compressible} if there exists a compressing disk for $F - T$ in $M - T$. If $F$ is not $T$--compressible, it is {\em  $T$--incompressible}. $F$ is {\em  $T$--$\bdd$--compressible} if there exists a disk $D \subset M - T$ with interior disjoint from $F$ such that $\bdd D$ is the endpoint union of an arc $\gamma$ in $F$ and an arc $\delta$ in $\bdd M$. We require that $\gamma$ not be parallel in $F - T$ to an arc of $\bdd F - T$. If $F$ is not $T$--$\bdd$--compressible, it is {\em  $T$--$\bdd$--incompressible}. Finally suppose $\Gamma$ is some subgraph of $T$.  We will say that $F$ is {\em  $\Gamma$-cut-compressible}, if there exists a disk $D^c$ that only intersects $F$ it its boundary, the boundary is essential in $F-T$, so that $|D^c \cap T|=1$ and that point of intersection is contained in $\Gamma$. We also require that $\boundary D^c$ is not parallel in $F - T$ to a puncture $T \cap F$. We call $D^c$ a {\em $\Gamma$-cut-disk}. A {\em $\Gamma$-c-disk} will be either a $T$-compressing disk or a $\Gamma$-cut-disk. If $\Gamma$ is understood, we refer to it as simply a c-disk. Figure \ref{fig:GammaCbdy} depicts, among other things, a cut disc. A surface $F$ in $M$ is called {\em $T$-parallel} if $F$ is boundary parallel in $M-\inter{\eta} (T)$ and {\em $T$-essential} if it is $T$-incompressible and not $T$-parallel. The surface $F$ will be called $T$-c-essential if it is essential and cut-incompressible.

\subsection{Trivially embedded graphs in compressionbodies}

Let $C$ be a compressionbody and $T$ be a properly embedded graph in $C$. A connected component $\tau$ of $T$ is {\em trivial} in $C$ if it is one of four types (see Figure \ref{fig:GammaCbdy}):
\begin{enumerate}
\item {\em Bridge arc:} a single edge with both endpoints in $\bdd_+ C$ which is parallel to an arc in $\bdd_+ C$. The disk of parallelism is called a {\em bridge disk}. 
\item {\em Vertical edge:} a single edge with one endpoint in $\bdd_+C$ and one endpoint in $\bdd_-C$ that is isotopic to $\{\text{point}\} \times I$.

\item {\em Pod:} a graph with a single vertex in the interior of $C$ and with all valence 1 vertices lying in $\bdd_+C$ so that there is a disk $D$ with $\bdd D \subset \bdd_+C$ inessential in $\bdd_+C$ and so that $\tau \subset D$. The disk $D$ is called a {\em pod disk}. Each of the components of $D-\tau$ will be called a {\em bridge disk} as these components play the same role as bridge disks for bridge arcs. 

\item {\em Vertical pod:} a graph with a single vertex in the interior of $C$ and with one valence 1 vertex lying in $\bdd_-C$ and all other valence 1 vertices in $\bdd_+C$ so that if the edge adjacent to $\bdd_- C$ is removed the resulting graph is a pod and if instead all but one of the edges adjacent to $\bdd_+C$ are removed, the result is a vertical edge with a valence 2 vertex in its interior. The edges that have one endpoint in $\bdd_+C$ are called {\em pod legs} and the other edge is called a {\em pod handle}. 
 \end{enumerate}
 
 If all components of $T$ are simultaneously trivially embedded, then we say that $T$ is {\em trivially embedded} in $C$.

\subsection{Trivially embedded graphs in $\Gamma$-compressionbodies}

\begin{defin} Let $C$ be a compressionbody containing a properly embedded graph $T$ and let $\Gamma$ be a subgraph of $T$. Suppose that there is a collection of $(T,\Gamma)$-cut-disks, $\mathcal{D}^c$, such that:
\begin{enumerate}
\item for each edge of $\Gamma$ there is at most one disk in $\mathcal{D}^c$ intersecting it,
\item each edge of $\Gamma$ intersected by $\mathcal{D}^c$ has both endpoints on $\boundary_- C$,
\item boundary reducing $C$ along all cut-disks $\mathcal{D}^c$ produces a union of compressionbodies $C_1,..,C_n$,
\item for each $i$ the graph $C_i \cap T$ is trivially embedded in $C_i$.

\end{enumerate}

Then we call the triple of the compressionbody and the graphs a {\em $\Gamma$-compressionbody containing a trivially embedded graph $T-\Gamma$} and denote it $(C,T,\Gamma)$. 

\end{defin}

Note that if $\Gamma=\emptyset$ then $(C,T,\Gamma)=(C,T)$ is a compressionbody containing a trivially embedded graph. 

If $(C,T,\Gamma)$ is a $\Gamma$-compressionbody and $\gamma$ is an arc of $\Gamma$ in $(C,T,\Gamma)$ that has both of its endpoint in $\bdd_-C$ and contains no vertices, then $\gamma$ is called a {\em ghost arc}. Each ghost arc intersects a $\Gamma$-cut disk in $C$.

\begin{figure}[tbh]
\centering
\includegraphics[scale=0.4]{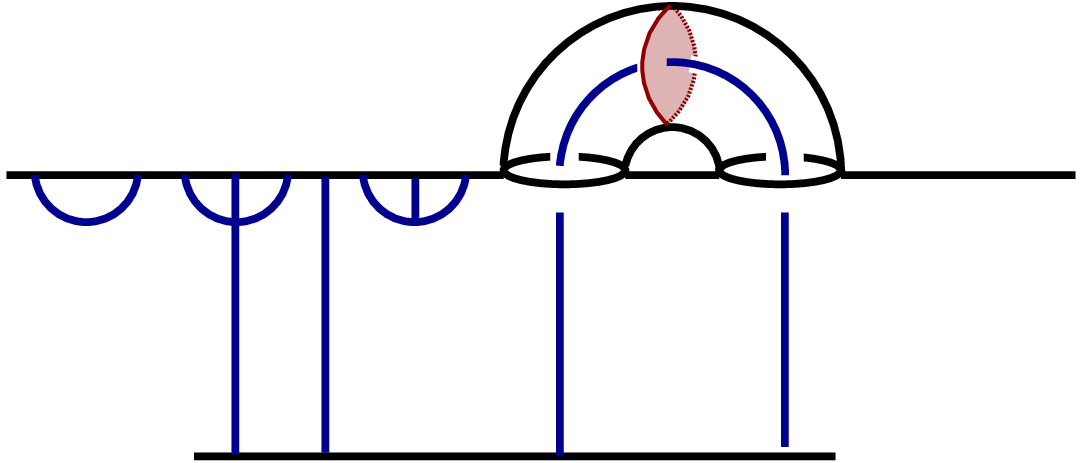}
%\scalebox{0.4}{\input{stabilization.eps_t}}
\caption{The five types of components in a $\Gamma$-compressionbody. From left to right we have: bridge arc, vertical pod, vertical arc, pod, ghost arc. The red disc is a cut disc for the ghost arc.} \label{fig:GammaCbdy}
\end{figure}

\subsection{Heegaard surfaces and $\Gamma$-Heegaard surfaces}\label{subsec:def}
Let $(M,T)$ be a (3-manifold, graph) pair. A {\em  Heegaard splitting} for $(M,T)$ is a decomposition of $M$ into two compressionbodies, $C_1$ and $C_2$, such that $T_i = T \cap C_i$ is trivially embedded in $C_i$ for $i \in \{1,2\}$. The surface $H=\bdd_+C_1=\bdd_+C_2$ is called a {\em Heegaard surface} for $(M,T)$. We will also say that $T$ is in {\em  bridge position} with respect to $H$ and that $H$ is a {\em  bridge splitting} of $(M,T)$. 

Suppose now that $\Gamma$ is a subgraph of $T$ and $H$ is a surface in $(M,T)$ transverse to $T$ so that $H$ splits $M$ into compressionbodies $C_1$ and $C_2$ such that $(C_i,T_i,\Gamma_i)$ is a $\Gamma_i$-compressionbody containing a trivially embedded graph $T_i-\Gamma_i$ for $i \in \{1,2\}$. In this case we say $H$ is {\em $\Gamma$-Heegaard surface} or a {\em $\Gamma$-bridge surface} for $(M,T,\Gamma)$. If $\Gamma=\emptyset$ then $H$ is simply a Heegaard surface.

Suppose $H$ is a $\Gamma$-Heegaard surface for $(M, T,\Gamma)$ splitting it into triples $(C_1, T_1,\Gamma_1)$ and $(C_2, T_2,\Gamma_2)$. We will say that $H$ is {\em $T$-reducible} if there exists a sphere $S$ disjoint from $T$ such that $S \cap H$ is a single curve essential in $H-T$, otherwise $H$ is {\em $T$-irreducible}. We will say that $H$ is {\em $T$-weakly reducible} if $H$ has $T$-compressing disks on opposite sides with disjoint boundaries. Otherwise $H$ is said to be {\em $T$-strongly irreducible}. We will say that $H$ is {\em $(T,\Gamma)$-c-weakly reducible} if $H$ has $\Gamma$-c-disks on opposite sides with disjoint boundaries. Otherwise $H$ is said to be {\em $(T,\Gamma)$-c-strongly irreducible}. 

\subsection{Multiple bridge splittings}

To prove our main theorem we will extend to graphs the definition of multiple bridge splittings for the pair (3-manifold, 1-manifold) introduced by Hayashi and Shimokawa in \cite{HS2} and generalized by Tomova in \cite{T1}. 

\begin{defin}
Suppose $M$ is a 3-manifold containing a properly embedded graph $T$ and let $\Gamma$ be a subgraph of $T$. A disjoint union of surfaces $\mathcal{H}$ is a multiple $\Gamma$-Heegaard splitting for $(M,T,\Gamma)$ if:
\begin{itemize}
\item the closure of each component of $(M, T)-\mathcal{H}$ is a $\Gamma$-compressionbody $C_i$ containing a trivially embedded graph $T_i=(T-\Gamma)\cap C_i$,
\item for each $i$, $\bdd_+C_i$ is attached to some $\bdd_+C_j$ and each component of $\bdd_-C_i$ is either contained in $\bdd M$ or is attached to a component of some $\bdd_-C_k$ with $i \neq k$. 
\end{itemize}
Let $\mathcal{H}^+=\cup \bdd_+C_i$ and $\mathcal{H}^-=\cup \bdd_-C_i$. The components of $\mc{H}^+$ and $\mc{H}^-$ are called {\em thick} and {\em thin} surfaces respectively.
\end{defin} 

Note that in our definition each thick surface is separating in the complement of the thin surfaces.

\subsection{Generalized stabilizations, perturbations and removable paths}

Several geometric operations can be used to produce new $\Gamma$-bridge surfaces from old ones. These are generalizations of stabilizations for Heegaard splittings of manifolds and usually we work with bridge surfaces that are not obtained from others via these operations. A more detailed discussion of these operations can be found in \cite{HS1,STo3,TT1}. A $\Gamma$-bridge surface $H$ for $(M,T, \Gamma)$ will be called {\em stabilized} if there is a pair of $T$-compressing disks on opposite sides of $H$ that intersect in a single point. The $\Gamma$-bridge surface is {\em meridionally stabilized} if there is a $(T, \Gamma)$-cut-disk and a $T$-compressing disk on opposite sides of $H$ that intersect in a single point.

As we are considering manifolds with boundary there are two other geometric operations that can be used to obtain a new bridge surface from an old one. Suppose $H$ is a $\Gamma$-Heegaard splitting for $(M,T, \Gamma)$ decomposing $M$ into compressionbodies $C_1$ and $C_2$. Let $F$ be a component of $\bdd_- C_1 \subset \bdd M$ and let $T'$ be a collection of vertical edges in $F \times [-1,0]$ so that $T' \cap (F \times \{0\}) = T \cap F$. Let $H'$ be a minimal genus Heegaard surface for $(F \times [-1,0], T')$ which does not separate $F \times \{-1\}$ and $F \times \{0\}$ and which intersects each edge in $T'$ exactly twice. $H'$ can be formed by tubing two parallel copies of $F$ along a vertical arc not in $T'$. We can form a $\Gamma$-Heegaard surface $H''$ for $M \cup (F \times [-1,0])$ by {\em amalgamating} $H$ and $H'$. This is simply the usual notion of amalgamation of Heegaard splittings (see \cite{Sc}). In fact, $H''$ is a $\Gamma$-Heegaard surface for $(M \cup (F \times [-1,0]), T \cup T')$. Since $(M \cup (F \times [-1,0]), T \cup T')$ is homeomorphic to $(M,T)$, we may consider $H''$ to be a $\Gamma$-Heegaard surface for $(M,T,\Gamma)$. $H''$ is called a {\em boundary stabilization} of $H$. A similar construction can be used to obtain a new $\Gamma$-Heegaard splitting of $(M,T,\Gamma)$ by tubing two parallel copies of $F$ along a vertical arc that does lie in $T'\cap \Gamma$. In this case $H''$ will be called {\em meridionally boundary stabilized}.

If a $\Gamma$-bridge surface is stabilized, boundary stabilized, meridionally stabilized, or meridionally boundary stabilized  we will say that it contains a {\em generalized stabilization}.

A $\Gamma$-bridge surface is called {\em cancellable} if there is a pair of 
bridge disks $D_i$ on opposite sides of $H$ such that $ \emptyset \neq (\bdd D_1 \cap \bdd D_2) \subset (H \cap T)$. If $|\bdd D_1 \cap \bdd D_2|=1$ we will call the bridge surface {\em perturbed}. Unlike the case when $T$ is a 1-manifold, a perturbed bridge surface cannot necessarily be unperturbed by an isotopy as that may result in edges that are no longer trivially embedded. The pair $\{D_1,D_2\}$ is called the {\em associated cancelling pair of disks}.

Suppose that $\zeta \subset T$ is a 1--manifold which is the the union of edges in $T$ (possibly a closed loop with zero or more vertices of $T$). We say that $\zeta$ is a {\em removable path} if the following hold:
\begin{enumerate}
\item Either the endpoints of $\zeta$ lie in $\bdd M$ or $\zeta$ is a cycle in $T$.

\item $\zeta$ intersects $H$ exactly twice 

\item If $\zeta$ is a cycle, there exists a cancelling pair of disks $\{D_1,D_2\}$ for $\zeta$ with $D_j \subset C_j$. Furthermore there exists a compressing disk $E$ for $H$ such that $|E \cap T| = 1$ and if $E \subset C_j$ then $|\bdd E \cap \bdd D_{j+1}| = 1$ (indices run mod 2) and $E$ is otherwise disjoint from a complete collection of bridge disks for $T - H$ containing $D_1 \cup D_2$. 

\item If the endpoints of $\zeta$ lie in $\bdd M$, there exists a bridge disk $D$ for the bridge arc component of $\zeta - H$ such that $D - T$ is disjoint from a complete collection of bridge disks $\Delta$ for $T - H$. Furthermore, there exists a compressing disk $E$ for $H$ on the opposite side of $H$ from $D$ such that $|E \cap D| = 1$ and $E$ is disjoint from $\Delta$.
\end{enumerate}
See Figure \ref{fig:removablepath} for an example of a removable path.

\begin{figure}[tbh]
\centering
\includegraphics[scale=0.4]{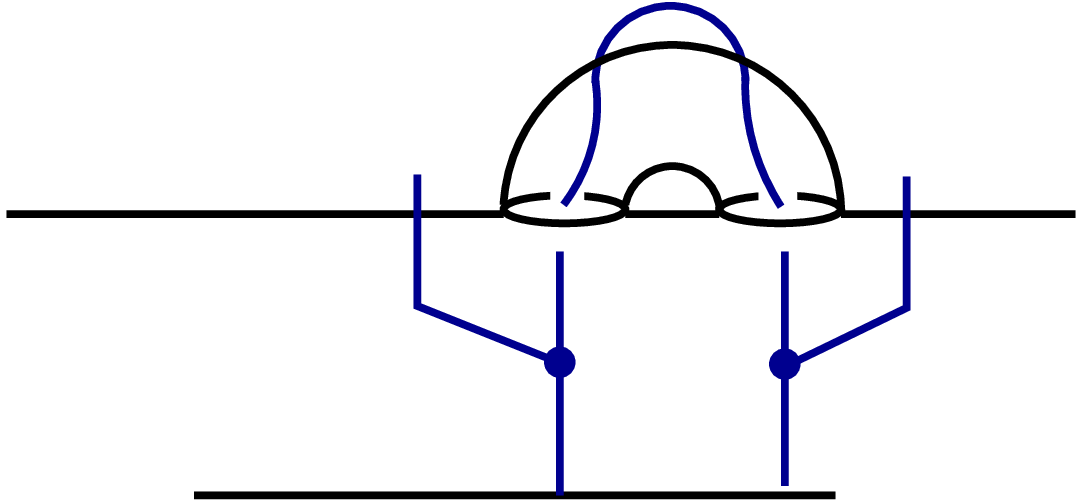}
%\scalebox{0.4}{\input{stabilization.eps_t}}
\caption{An example of a removable path.} \label{fig:removablepath}
\end{figure}

If $T$ has a removable path $\zeta$, $\zeta$ can be isotoped to lie in a spine for one of the compressionbodies $C_1$ or $C_2$. A condition weaker than being removable is the following: A circle or edge component $\gamma$ of $T$ is {\em isotopically core} if it can be isotoped in $M - (T - \gamma)$ to lie in the 1--dimensional portion of a spine of one of the compressionbodies $C_1$ or $C_2$. (Recall that the spine of a compressionbody is the union of the negative boundary of the compressionbody with a 1--complex such that the exterior of the spine is a product region.)

\subsection{$\Gamma$-c-Heegaard surfaces and removable edges}
This section presents a technical result which sometimes allows a $\Gamma$-c-Heegaard surface to be converted into a Heegaard surface with removable edges. 

Suppose that $H$ is a $\Gamma$-Heegaard surface for $(M,T, \Gamma)$ and that $e \subset \Gamma$ is an edge disjoint from $H$ with both endpoints in $\boundary M$. Let $E$ be a cut disk intersecting $e$ whose boundary is in $H$. By isotoping $e$ so that $e \cap E$ moves through $\boundary E$ to the other side, we convert $e$ into a removable path $e'$. Let $T'$ be the new graph. Let $H'$ be the new $\Gamma$-Heegaard surface for $(M,T')$. Let $D$ be the bridge disk for the bridge arc component of $e' - H'$. See Figure \ref{fig:TechnicalResult}.

\begin{figure}[tbh]
\labellist \small\hair 2pt 
\pinlabel {$\boundary M$} [b] at 132 10
\pinlabel {$e$} [r] at 77 61
\pinlabel {$e'$} [r] at 437 61
\pinlabel {$E$} [l] at 137 176
\pinlabel {$H$} [l] at 265 139
\pinlabel {$H'$} [l] at 623 139
\pinlabel {$D$} at 494 236
\endlabellist 
\centering 
\includegraphics[scale=0.5]{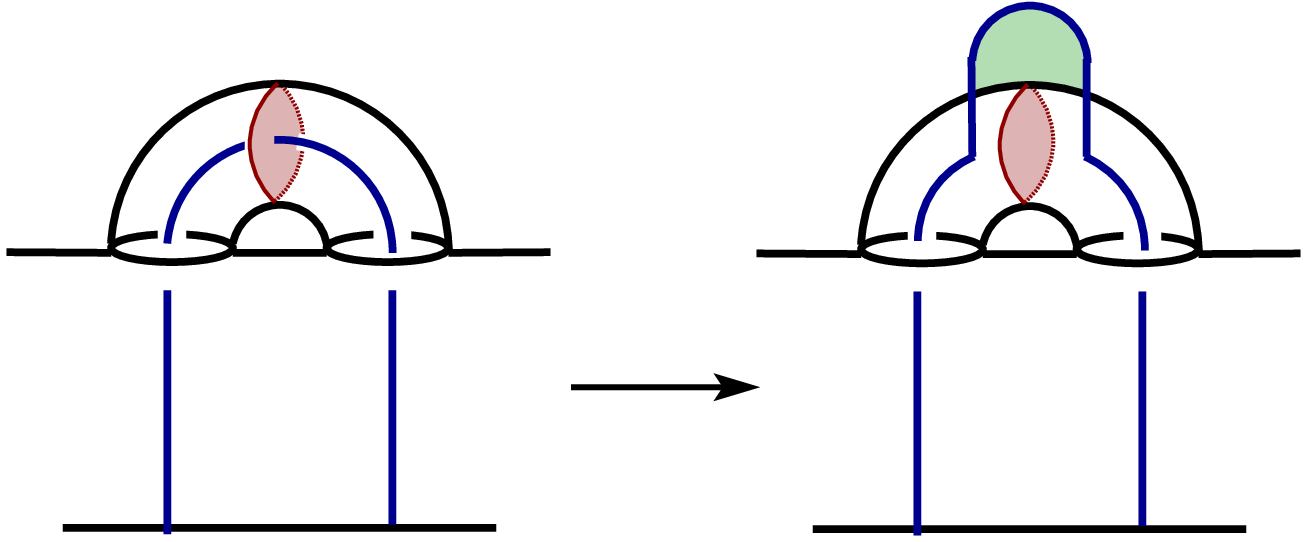}
\caption{The conversion of a ghost edge into a removable edge.} \label{fig:TechnicalResult}
\end{figure}
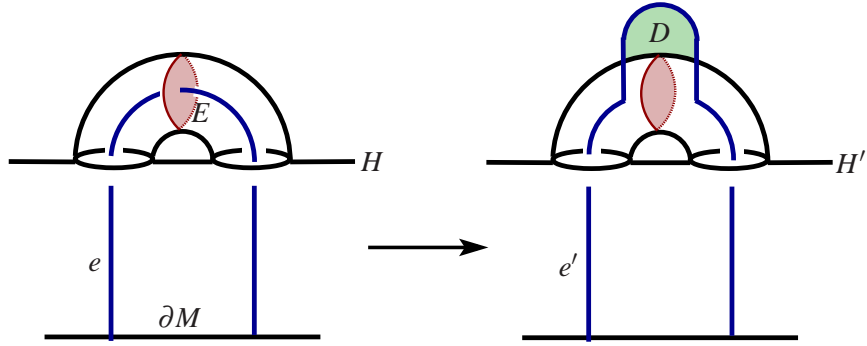

\begin{lemma}\label{lem: creating removable edges}
If $H'$ is stabilized or meridionally stabilized then so is $H$. If $H'$ is boundary-stabilized or meridionally boundary stabilized, so is $H$ and the stabilization is along the same component of $\bdd M$. If $H'$ is perturbed then so is $H$. If $T'$ contains a removable path other than $e'$ then either that path is removable in $T$ or $H$ is meridionally stabilized. 
\end{lemma}
\begin{proof}
\textbf{Case 1:} Suppose that $H'$ is stabilized or meridionally stabilized by disks $D_1$ and $D_2$ on opposite sides of $H'$ which intersect once. Out of all such pairs of stabilizing or meridionally stabilizing pairs, choose $D_1$ and $D_2$ so that $|(D_1 \cup D_2) \cap (D \cup E)|$ is minimal. The next claim shows that $(D_1,D_2)$ also stabilizes or meridionally stabilizes $H$.

\textbf{Claim:} $|(D_1 \cup D_2) \cap e| \leq 1$.

Since $(D_1,D_2)$ (meridionally) stabilize $H'$, $|(D_1 \cup D_2) \cap e'| \leq 1$. Without loss of generality, choose the labelling so that $D_1$ is on the same side of $H$ as $D$. Then $D_1 \cap e = \nil$, and so $(D_1 \cup D_2) \cap e = D_2 \cap e$. If $D_2$ is disjoint from and not parallel to $E$, then clearly $|e \cap D_2| \leq 1$. If $D_2$ is parallel to $E$, then $|D_2 \cap e| = |E \cap e| = 1$. If $D_2$ is not disjoint from $E$, then by the minimality of $|(D_1 \cup D_2) \cap (E \cup D)|$ it follows that $D_2 \cap E$ is a collection of arcs. After possibly a small isotopy, we may assume that $D \cap E$ is disjoint from $D_2 \cap E$. Then using $D$ to isotope $e'$ back to $e$ guarantees that $e$ is disjoint from $D_2$. \qed(Claim)

\textbf{Case 2:} Suppose that $H'$ is boundary stabilized or meridionally boundary stabilized. In this case there exists a $\Gamma$-c-disk $D'$ for $H$ such that compressing $H'$ along $D'$ produces surfaces $H_1$ and $H_2$ and the surface $H_2$ bounds a product region with $\boundary M$ containing only vertical arcs of $T$ while the surface $H_1$ is a $\Gamma$-Heegaard surface for $(M,T, \Gamma)$. 

If $D'$ is disjoint from or parallel to $E$ then it is clear that $H$ is boundary stabilized or meridionally boundary stabilized. Suppose, therefore that $D'$ intersects $E$. We may assume that $D'$ was chosen so as to minimize $D' \cap E$. This implies that $D' \cap E$ is a non-empty collection of arcs. Then, as in Case 1, isotoping $e'$ back to $e$ shows that $e \cap D' = \nil$. Hence, $H$ is boundary stabilized or meridionally boundary stabilized and the stabilization is along the same component as that of $H'$.

\textbf{Case 3:} Suppose that $(D_1,D_2)$ are a perturbing pair of disks for $H'$ such that $D_1$ is on the same side of $H$ as $D$. Notice that $\boundary D_1 \cup \boundary D_2$ is disjoint from $e'$ since two components of $e' - H'$ are vertical arcs in the compressionbody containing them. Thus, unless $e \cap D_2 \neq \nil$, $(D_1, D_2)$ is a perturbing pair for $H$. If $D_2 \cap e \neq \nil$, then $D_2$ intersects the neighborhood of $E$ used to push $e$ to $e'$. An argument similar to that of Cases 1 and 2 shows that $\boundary D$ can be assumed to be disjoint from $\boundary D_2$, and so $e$ is disjoint from $D_2$, as desired.

\textbf{Case 4:} Suppose that $T'$ contains a removable path $\zeta \neq e$. An argument similar to the previous cases shows that $\zeta$ is a removable path in $T$, unless $\zeta$ is not a cycle and the compressing disk $E$ from condition (4) of the definition of removable path is equal to the present disk $E$. Suppose, therefore, that this is the case. By the definition of removable path, there is a bridge disk $D'$ for $e' - H'$ which is disjoint from $E$. A small isotopy of $D \cup D'$ creates a compressing disk $E'$ for $H' - T'$ intersected once by $E$. The pair $(E,E')$ therefore, shows that $H'$ is stabilized. Since $\boundary E'$ is non-separating on $H'$ and therefore on $H$, $E'$ is a compressing disk for $H$ disjoint from $T$. The boundary of $E'$ intersects $E$ exactly once and $E \cap T = E \cap e$ and so $H$ is meridionally stabilized.
\end{proof}

\section{Properties of compressionbodies containing properly embedded graphs}

In this section we will generalize many of the well known results for compressionbodies to the case when the compressionbody contains a graph embedded in a specific way.

\begin{lemma}\label{lem:compressingcompressionbodies}
Suppose $C$ is a $\Gamma$-compressionbody containing a trivially embedded graph $T-\Gamma$. Compressing or cut-compressing $C$ results in a union of $\Gamma$-compressionbodies each containing a trivially embedded graph. 

\end{lemma}
\begin{proof}
Given any $\Gamma$-c-disk $D^c$ for $(C,T,\Gamma)$ we can always find a collection of pairwise disjoint $\Gamma$-c-disks $\mathcal{D}^c$ containing $D^c$ so that $(C,T,\Gamma)$ c-compressed along $\mathcal{D}^c$ is a collection $3$-balls and components homeomorphic to $G \times I$ where $G$ is a component of $\bdd_-C$. Both types of components may contain trivially embedded graphs. The result follows. 
\end{proof}

The next lemma shows that the negative boundary of a $\Gamma$-compressionbody is $\Gamma$-c-incompressible.

\begin{lemma} \label{lem:negativeboundaryincomp}
Let $(C,T,\Gamma)$ be a $\Gamma$-compressionbody containing a trivially embedded graph $T-\Gamma$. Assume that no component of $\boundary_- C$ is a 2--sphere intersecting $T$ one or fewer times. Then $C - T$ is irreducible, no sphere in $C$ intersects $T$ exactly once, and $\bdd_- C$ is $\Gamma$-c-incompressible.
\end{lemma}
\begin{proof}
Suppose that $S \subset C$ is either a sphere intersecting $T$ transversally one or fewer times and which does not bound a ball in $C - T$ or a $\Gamma$-compressing disk for $\boundary_- C$. An innermost disk argument shows that we may choose $S$ so that it is disjoint from all the $\Gamma$-c-disks for $\boundary_+ C$. Let $\Delta$ be a maximal collection of $\Gamma$-c-disks for $\boundary_+ C$, chosen so that each disk of $\Delta$ intersecting $T$ intersects a unique ghost arc of $T$ exactly once and so that each ghost arc of $T$ intersects $\Delta$ exactly once. Boundary-reducing $C$ along $\Delta$ creates a $\Gamma$-compressionbody $(C',T',\Gamma')$ that is possibly disconnected. $C'$ is the union of trivial compressionbodies and $S \subset C'$. No component of $T'$ is a ghost arc since every ghost arc intersects a disk in $\Delta$. In a trivial compressionbody, there are no essential spheres or disks. Hence, if $S$ is a sphere, it either bounds a 3--ball in $C'$ or is boundary parallel in $C'$.

If $S$ is a sphere bounding a 3--ball in $C'$, then since every component of $T'$ is adjacent to $\boundary_+ C'$, $S$ bounds a 3--ball in $C' - T'$. Hence if $S$ bounds a 3--ball in $C'$ it also bounds one in $C - T$, a contradiction. Thus, if $S$ is a sphere, then it is in a component $C''$ of $C'$ homeomorphic to $S^2 \times [0,1]$ and is boundary-parallel. At least two components of $T \cap C''$ are adjacent to $\boundary_- C \cap C''$ and these components are also adjacent to $\boundary C'' - \boundary_- C$. Hence these components also intersect $S$ and so $S$ intersects $T$ twice, a contradiction. 

If $S$ is a disk, then $\boundary S$ bounds a disk $E \subset \boundary_- C = \boundary_- C'$. If $S$ is a compressing disk, then $E$ is adjacent to at least one component of $T'$ and if $S$ is a cut disk, then $E$ is adjacent to at least two components of $T'$. The union $S \cup E$ is a sphere in $C'$ which either bounds a 3--ball in $C'$ or is boundary-parallel. If the sphere bounds a 3-ball, then each intersection point of $T' \cap E$ is joined in $T'$ to an intersection point of $T' \cap S$, implying that $S$ intersects $T'$ at least once if it is a compressing disk and at least twice if it is a cut disk. This contradicts the definition of compressing disk and cut disk, respectively. Suppose the sphere is parallel to a boundary component $P$ of $C'$.  If $P \subset \boundary_+ C'$, then $E$ is a subset of a sphere component $P'$ of $\boundary_- C'$. The disk $E' = P - E$ must intersect $T'$ once or twice (corresponding to $S$ being a compressing or cut disk). Since the positive boundary of a connected compressionbody is connected, the sphere $S \cup E'$ bounds a 3--ball in $C'$. This gives rise to a contradiction as before. Hence, $P \subset \boundary_- C'$. Since no component of $T'$ is adjacent to two components of $\boundary_- C'$, each intersection point of $T' \cap E$ corresponds to an intersection point of $T' \cap S$. Hence, $D$ can be neither a compressing or a cut disk, a contradiction.
\end{proof}

The next lemma does not concern bridge surfaces \textit{per se}, but will be useful in the proof of the main theorem when we try to show that the thin surfaces are not $T$-parallel. We use it to guarantee that if there is a thin surface that is not $T$-parallel, then we can find an innermost such surface.

\begin{lemma}\label{lem:essential}
Suppose that $C$ is a compressionbody with $F = \boundary_+ C$. Let $T \subset C$ be a properly embedded graph such that $F - \inter{\eta}(T)$ is parallel to $\boundary (C - \inter{\eta}(T)) - F$. Let $G \subset C$ be a $T$-incompressible surface disjoint from $F$ and transverse to $T$. Then each component of  $G - \inter{\eta}(T)$ is parallel to  $\boundary (C - \inter{\eta}(T)) - F$. 
\end{lemma}

\begin{proof}
Let $N = C - \inter{\eta}(T)$. Then $N$ is homeomorphic to $(F - \inter{\eta}(T)) \times I$ with the homeomorphism taking $F - \inter{\eta}(T)$ to $(F - \inter{\eta}(T)) \times \{0\}$. The boundary of $G- \inter{\eta}(T)$ is contained in $(F - \inter{\eta}(T)) \times \{1\}$. Since $G$ is $T$-incompressible, by \cite[Corollary 3.2]{W2}, each component of  $G - \inter{\eta}(T)$ is parallel to $(F - \inter{\eta}(T)) \times \{1\}$. Hence, each component of  $G - \inter{\eta}(T)$ is parallel to  $\boundary (C - \inter{\eta}(T)) - F$.
\end{proof}
  
Finally we recall the classification of Heegaard splittings of pairs $(M,T)$ where $M$ is a compressionbody and $\boundary_+ M - \inter{\eta}(T)$ is parallel to $\boundary (M - \inter{\eta}(T)) - \boundary_+ M$, \cite{TT1}. This result is key to proving our main theorem. 

\begin{thm}\cite[Theorem 3.1]{TT1} \label{thm:HSofCompBodies}
Let $M$ be a compressionbody and $T$ be a properly embedded graph so that $\boundary_+ M - \inter{\eta}(T)$ is parallel to $\boundary (M - \inter{\eta}(T)) - \boundary_+ M$. Let $H$ be a Heegaard surface for $(M,T)$. Assume that $T$ contains at least one edge. Then one of the following occurs:
\begin{enumerate}
\item $H$ is stabilized;
\item $H$ is boundary stabilized;
\item $H$ is perturbed;
\item $T$ has a removable path disjoint from $\boundary_+ M$; 
\item $M$ is a 3--ball, $T$ is a tree with a single interior vertex (possibly of valence 2), and $H - \inter{\eta}(T)$ is parallel to $\bdd M - \inter{\eta}(T)$ in $M - \inter{\eta}(T)$;
\item $M = \bdd_- M \times I$, $H$ is isotopic in $M - \inter{\eta}(T)$ to $\bdd_+ M - \inter{\eta}(T)$.
\end{enumerate}
\end{thm}

For our purposes we need to strengthen the second conclusion of the above theorem:

\begin{thm} \label{thm:HSofCompBodies2} Let $M$ be a compressionbody and $T$ be a properly embedded graph so that $\boundary_+ M - \inter{\eta}(T)$ is parallel to $\boundary (M - \inter{\eta}(T)) - \boundary_+ M$. Let $H$ be a Heegaard surface for $(M,T)$. Assume that $T$ contains at least one edge. Then one of the following occurs:
\begin{enumerate}
\item $H$ is stabilized;
\item $H$ is boundary stabilized along $\bdd_-M$;
\item $H$ is perturbed;
\item $T$ has a removable path disjoint from $\boundary_+ M$; 
\item $M$ is a 3--ball, $T$ is a tree with a single interior vertex (possibly of valence 2), and $H - \inter{\eta}(T)$ is parallel to $\bdd M - \inter{\eta}(T)$ in $M - \inter{\eta}(T)$;
\item $M = \bdd_- M \times I$, $H$ is isotopic in $M - \inter{\eta}(T)$ to $\bdd_+ M - \inter{\eta}(T)$.
\end{enumerate}

\end{thm}
\begin{proof}
Suppose $H$ is boundary stabilized along $\bdd_+M$. Then $H$ is obtained by amalgamating a minimal genus Heegaard surface for $\bdd_+M \times [-1,0]$ which does not separate $\bdd_+M \times \{-1\}$ and $\bdd_+M \times \{0\}$ and which intersects each edge in $T\cap (\bdd_+M \times [-1,0])$ exactly twice, together with a Heegaard surface $\tilde{H}$ for $M$. Without loss of generality we will assume that $H$ is obtained from $\tilde{H}$ after a single boundary stabilization along $\bdd_+M$.

By Theorem \ref{thm:HSofCompBodies}, $\tilde{H}$ satisfies one of six possible conclusions. If $\tilde{H}$ is stabilized, perturbed or if $T$ has a removable path disjoint from $\boundary_+ M$ then the same is true for $H$ as boundary stabilizations preserve all of these properties. If $\tilde{H}$ is boundary stabilized, the stabilization must be along $\bdd_-M$, and so the same is true for $H$. 

Suppose that $M$ is a 3--ball, $T$ is a tree with a single interior vertex, and $H - \inter{\eta}(T)$ is parallel to $\bdd M - \inter{\eta}(T)$ in $M - \inter{\eta}(T)$. Let $A$ and $B$ be the components of $M-\tilde{H}$ so that $B$ is a ball and let $\kappa$ be one of the edges of $T \cap M$. Then $H$ can be recovered from $\tilde{H}$ by tubing $H$ to the boundary of a collar of $\bdd M$ along a vertical tube $\tau$ in $A$. We can choose $\tau$ to be arbitrarily close to $\kappa \cap A$; in particular, we may assume that the disk of parallelism between $\tau$ and $\kappa$ intersects some bridge disk that contains $\kappa \cap B$ in its boundary only in the point $\kappa \cap \tilde{H}$. We conclude that $H$ is perturbed.

Suppose then that $M= \bdd_- M \times I$ and $H$ is isotopic in $M - \inter{\eta}(T)$ to $\bdd_+ M - \inter{\eta}(T)$. Let $A$ and $B$ be the components of $M-\tilde{H}$ so that $A$ contains $\bdd_+M$. The argument in this case is identical to the one above as long as there is at least one bridge disk in $B$. If $T \cap B$ is a product, then $T \cap M$ is a collection of vertical arcs and thus $H$ can be obtained from the Heegaard surface $\tilde{H}$ by stabilizing along $\bdd_-M$. 
\end{proof}

\section{Haken's lemma}  
  
  Let $H$ be a $\Gamma$-bridge surface for $(M,T, \Gamma)$ and suppose $D$ is a $T$-compressing disk for some component $G$ of $\bdd M$. It is a classic result of Haken \cite{Hk} that in the case $T=\emptyset$, there is a compressing disk $D'$ for $G$ so that $D'$ intersects $H$ in a unique essential curve. This result was extended to the case where $T$ is a 1-manifold and $\Gamma=\emptyset$ in \cite{HS2} and to the case where $T$ is a 1-manifold and $\Gamma=T$ in \cite{T1}. This result implies that, in most cases, if $\bdd M$ has a compressing disk, then $H$ is weakly reducible. 
  
  Here we will use a generalization of this result to cut-disks. We are grateful to Jesse Johnson for showing us a proof of the following theorem.
  
    \begin{thm} \label{thm:one}
Suppose $M$ is a compact orientable manifold, $T$ is a properly embedded graph in $M$ and $\Gamma$ is a subgraph of $T$. Assume that $M - T$ is irreducible and that no sphere in $M$ intersects $T$ exactly once transversally. Suppose $H$ is a $\Gamma$-Heegaard splitting for $(M,T,\Gamma)$ such that there is a $\Gamma$-c-disk for $H$ on each side. If there exists a $\Gamma$-c-disk for some component of $\bdd M$ then $H$ is $\Gamma$-c-weakly reducible. 
   \end{thm}   
  
  \begin{proof}
Let $D^*$ be a $\Gamma$-c disk for $\bdd M$.

By Lemma \ref{lem:negativeboundaryincomp}, the negative boundary of a $\Gamma$-c-compressionbody is $(T, \Gamma)$-c-incomressible so $D^*$ must intersect $H$.  Let $D_a$ and $D_b$ be c-disks for $H$ above and below it. Let $H_a$ be isotopic to $H$ but with the c-disk $D_a$ shrunk to a small disk neighborhood of a point. In particular in $H_a$ we may assume that $D_a$ is disjoint from $D^*$. Furthermore we may assume that $H_a$ is isotoped so that no curve of $D^* \cap H_a$ bounds a disk in $H_a - T$. As $H_a$ and $D^*$ must intersect, an innermost in $D^*$ circle of intersection bounds a $\Gamma$-c-disk $\Delta_a$ for $H_a$. The disk $\Delta_a$ is disjoint from $D_a$ and therefore either we are done, or $\Delta_a$ is on the same side of $H_a$ as $D_a$. Similarly there is a surface $H_b$ isotopic to $H$ so that an innermost curve of intersection between $H_b$ and $D^*$ bounds a c-disk for $H_b$ on the same side as $D_b$. 

Now consider the isotopy $\rho$ between $H_a$ and $H_b$. The surface $H_a=\rho^{-1}(-1)$, intersects $D^*$ in a simple closed curve which gives a $\Gamma$ c-disk for $H$ on one side and the surface $H_b=\rho^{-1}(1)$, intersects $D^*$ in a simple closed curve which gives a $\Gamma$ c-disk for $H$ on the opposite side. As $\rho^{-1}(t)$ always intersects $D^*$ in at least one essential curve there are two cases to consider:

\begin{enumerate}
\item There is some regular value $r$ where $\rho^{-1}(r)$ intersects $D^*$ in two disjoint essential curves $\sigma_a$ and $\sigma_b$ so that the interiors of the disks these curves bound on $D^*$ intersect
  $\rho^{-1}(r)$ only in inessential curves and near their boundaries these disks lie on opposite sides of $H$.
  
\item There is a critical point $c$ so that $\rho^{-1}(c-\epsilon) \cap D^*$ contains an essential curve with the properties of $\sigma_a$ and $\rho^{-1}(c+\epsilon) \cap D^*$ contains an essential curve with the properties of $\sigma_b$. 
\end{enumerate}

 As the curves of intersection in $\rho^{-1}(c-\epsilon) \cap D^*$ are disjoint from the curves $\rho^{-1}(c+\epsilon) \cap D^*$ in either case after an isotopy of $D^*$ to remove any inessential circles of intersection we can obtain c-disks $D_a$ and $D_b$ with disjoint boundaries $\sigma_a$ and $\sigma_b$ that lie on opposite sides of $H$ as desired.
  \end{proof}
 
 \section{Multiple $\Gamma$-bridge splittings of $(M,T, \Gamma)$}  

\subsection{Complexity}
We define a complexity on the thick surfaces $\mc{H}^+$ of a multiple $\Gamma$-Heegaard splitting of $(M,T,\Gamma)$.

\begin{defin}\label{defin:multiset ordering} Let $X$ be a set with an order $\leq$. Let $Y$ and $Z$ be two finite multisets of elements of $X$. Write $Y = (y_1, y_2, \hdots, y_n)$ and $Z = (z_1, z_2, \hdots, z_m)$ so that for all $i$, $y_i \geq y_{i+1}$ and $z_i \geq z_{i+1}$. We say that $Y < Z$ if and only if one of the following occurs:
\begin{itemize}
\item There exists $j \leq \min(n,m)$ so that for all $i < j$, $y_i = z_i$ and $y_j < z_j$.
\item $n < m$ and for all $i \leq n$, $y_i = z_i$.
\end{itemize}
\end{defin}

\begin{defin} Let $S$ be a closed connected surface 
embedded in $M$ transverse to a properly embedded graph $T
\subset M$. The complexity of $S$ is $c(S)=4g(S) + |S \cap T|$. 
\end{defin}

If $\mathcal{H}$ is a $\Gamma$-multiple bridge splitting for $(M,T, \Gamma)$, let the complexity of $\mathcal{H}$, denoted $c(\mathcal{H})$, be the multiset $\{c(S)| S \in \mathcal{H}^+\}$. If $\mathcal{H}$ and $\mathcal{H}'$ are two multiple $\Gamma$-Heegaard splittings for $(M,T, \Gamma)$, their complexities will be compared as in Definition \ref{defin:multiset ordering}.

The next lemma is immediate from the definition of complexity.

\begin{lemma}\label{lem: removing reduces}
Suppose that $\mc{H}$ is a multiple $\Gamma$-Heegaard splitting of $(M,T, \Gamma)$ and suppose that $\mc{J}$ is a multiple $\Gamma$-Heegaard splitting of $(M,T, \Gamma)$ such that $\mc{J}^+$ is a proper subset of $\mc{H}^+$. Then the complexity of $\mc{J}$ is strictly less than the complexity of $\mc{H}$.
\end{lemma}

Next, we will show that c-compressing a surface always decreases its complexity
\begin{lemma}\label{lem:compreduces}
Suppose $S$ is meridional surface in $(M,T)$ of non-positive euler characteristic. If $S'$ is a component of the surface obtained from $S$ by compressing along 
a c-disk, then $c(S)>c(S')$.
\end{lemma}

\begin{proof}
Suppose $S'$ is obtained from $S$ after a single c-compression along a c-disk $D^*$. If $D^*$ is non-separating then $g(S')=g(S)-1$ and $|S\cap T|\geq|S'\cap T|-2$ so $c(S)\geq c(S')-2$ and thus the complexity has strictly decreased.
If $D^*$ is separating then $S'$ has two components, $S'_1$ and $S'_2$, and $g(S')=g(S'_1)+g(S'_1)$ and $|S\cap T|\geq|S'_1\cap T|+|S'_2\cap T|-2$. As $D^*$ is a c-disk each of $S'_1$ and $S'_2$ must either have positive genus or at least 4 punctures. Therefore each of $S'_1$ and $S'_2$ either has genus strictly less than $g(S)$ or has at least 2 fewer punctures than $S$. In either case we can conclude that $c(S)>c(S'_1)$ and  $c(S)>c(S'_2)$ as desired.

\end{proof}

\begin{rmk}
In place of $c(S)$, we can use any complexity function that strictly decreases when a surface is $c$-compressed or when the number of intersections between the graph and the surface is decreased. For example, Tomova's complexity in \cite{T1} can also be used.
\end{rmk}

\section{Untelescoping}

We are interested in obtaining a multiple $\Gamma$-Heegaard splitting of $(M,T, \Gamma)$ with the property that every thin surface is $\Gamma$-c-incompressible. The following lemma will be useful and follows directly from Theorem \ref{thm:one}.

\begin{lemma}\label{lem: compressible thin surface}
Suppose that $\mc{H}$ is a multiple $\Gamma$-Heegaard splitting of $(M,T, \Gamma)$. If some component of $\mc{H}^-$ is $\Gamma$-c-compressible in $M$ then one of the following occurs:
\begin{itemize}

\item Some component of $\mc{H}^+$ is a $(T,\Gamma)$-c-weakly reducible $\Gamma$-Heegaard splitting of the component of $M - \mc{H}^-$ containing it; or
\item Some component of $M - \mc{H}$ is a product compressionbody with one boundary component in $\mc{H}^-$ and the other in $\mc{H}^+$. The graph $T$ intersects this component in vertical arcs and vertical pods with pod handles in $T - \Gamma$.
\end{itemize}
\end{lemma}

In the first case one can reduce the complexity of the splitting by c-compressing some thick surface along a pair of disjoint c-disks on opposite sides. This operation is called {\em untelescoping} and was first introduced by
 Scharlemann and Thompson in \cite{ST1} in the context of 3-manifolds. The concept was generalized to weakly reducible bridge surfaces for a manifold containing a properly embedded tangle in \cite{HS2} and then to $(\tau,\tau)$-c-weakly reducible $\tau$-bridge surfaces for a manifold containing a properly embedded tangle $\tau$ in \cite{T1}. In the following definition we extend the construction to $(T,\Gamma)$-c-weakly reducible $\Gamma$-bridge surfaces for $(M,T, \Gamma)$ where $T$ is a properly embedded graph. 

Let $H$ be a $(T,\Gamma)$-c-weakly reducible $\Gamma$-bridge splitting of $(M, T, \Gamma)$  and let $\mathcal{D}_1$ and $\mathcal{D}_2$ be collections of pairwise disjoint $\Gamma$-c-disks above and below $H$ such that $\mathcal{D}_1 \cap \mathcal{D}_2=\emptyset$. Then we can obtain a multiple $\Gamma$-bridge splitting for $(M, T)$ with one thin surface obtained from $H$ by $(T,\Gamma)$-c-compressing it along $\mathcal{D}_1\cup\mathcal{D}_2$ and two thick surfaces, obtained by $(T,\Gamma)$-c-compressing $H$ along $\mathcal{D}_1$ and $\mathcal{D}_2$ respectively. Although this operation is similar to the one described in \cite{T1} we include the details here with the modifications required.

Let $(A, A\cap T)$ and $(B, B\cap T)$ be the two $\Gamma$-compressionbodies 
into which $H$ decomposes $(M, T,\Gamma)$ and let $\mathcal{D}_A \subset A$ and $\mathcal{D}_B\subset B$ be collections of pairwise disjoint $\Gamma$-c-disks such that $\mathcal{D}_A \cap \mathcal{D}_B=\emptyset$. Let $A' = A-\inter{\eta}(\mathcal{D}_A)$ and $B'
=B-\inter{\eta}(\mathcal{D}_B)$. Then by Lemma \ref{lem:compressingcompressionbodies} $A'$ and $B'$ are each the disjoint union of $\Gamma$-compressionbodies containing trivial graphs $A' \cap T$ and $B' \cap T$ respectively.  

Take small collars $\eta(\bdd_+A')$ of
$\bdd_+A'$ and $\eta(\bdd_+B')$ of $\bdd_+B'$.  Let
$C^1=cl(A'-\eta(\bdd_+A'))$, $C^2=\eta(\bdd_+A')\cup \eta(\mathcal{D}_B)$,
$C^3=\eta(\bdd_+B')\cup \eta(\mathcal{D}_A)$ and
$C^4=cl(B'-\eta(\bdd_+B'))$.  Note that $C_1$ and $C_4$ are $\Gamma$-compressionbodies containing trivial graphs because they are homeomorphic to $A'$ and $B'$ respectively. $C_2$ and $C_3$ are obtained by taking $\text{surface} \times I$ containing vertical arcs and attaching 2-handles, some of which may contain segments of $\Gamma$ as their cores. Therefore $C_2$ and $C_3$ are also $\Gamma$-compressionbodies containing trivial graphs. We conclude that we have obtained a multiple $\Gamma$-bridge
splitting $\mathcal{H}$ of $(M,T, \Gamma)$ with thick surfaces $\bdd_+ C_1$ and $\bdd_+ C_2$ that can be obtained from $H$ by $\Gamma$-c-compressing along $\mathcal{D}_A$ and $\mathcal{D}_B$ respectively and a thin surface $\bdd_- C_2= \bdd_-C_3$ obtained from $H$ by $\Gamma$-c-compressing along both sets of c-disks. We say that $\mc{H}$ is obtained by untelescoping $H$ using $\Gamma$-c-disks. The next remark follows directly from Lemma \ref{lem:compreduces}.

\begin{rmk}
Suppose $\mathcal{H'}$ is a multiple $\Gamma$-bridge
splitting of $(M,T, \Gamma)$ obtained from another multiple $\Gamma$-bridge
splitting $\mathcal{H}$ of $(M,T, \Gamma)$ via untelescoping. Then $c(\mathcal{H'}) < c(\mathcal{H})$.
\end{rmk}

Suppose that $H$ is a multiple $\Gamma$-Heegaard splitting for $(M,T, \Gamma)$ with a thick surface that can be untelescoped to become a multiple $\Gamma$-Heegaard splitting $\mc{H}$ for $(M,T, \Gamma)$. It is natural to ask: If $\mc{H}$ contains a generalized stabilization must $H$ contain a generalized stabilization? If $\mc{H}$ contains a perturbed thick surface must $H$ also be perturbed? If $T$ has a path which is removable with respect to $\mc{H}$, is that path removable with respect to $H$? The answer is positive in all cases.
  
\begin{lemma}\label{lem:stabilization}
 Let $H$ be a (multiple) $\Gamma$-Heegaard splitting for the triple $(M,T, \Gamma)$ obtained by amalgamating the multiple $\Gamma$-Heegaard splitting $\mc{H}$ for $(M,T, \Gamma)$. (We assume that the amalgamation actually produces a (multiple) $\Gamma$-Heegaard splitting which is not necessarily the case). Then the following hold
 \begin{enumerate}
 \item Suppose the $\Gamma$-Heegaard splitting induced by $\mathcal{H}^+$ on some component of $(M,T, \Gamma)-\mathcal{H}^-$ 
contains a generalized stabilization. Furthermore suppose that if the generalized stabilization is a (meridional) boundary stabilization, then it is along a component of $\bdd M$. Then $H$ contains a generalized stabilization of the same type and if the generalized stabilization is a (meridional) boundary stabilization then it must also be along a component of $\bdd M$.
\item If the $\Gamma$-Heegaard splitting induced by $\mathcal{H}^+$ on some component of $(M,T,\Gamma)-\mathcal{H}^-$ 
is perturbed, then so is $H$.
\item If the $\Gamma$-Heegaard splitting induced by $\mathcal{H}^+$ on some component of $(M,T,\Gamma)-\mathcal{H}^-$ 
contains a removable path so that the path is either a cycle or both of its endpoints are contained in $\bdd M$, then so does $H$.
\item If $T$ contains a circle or edge component $\gamma$ that is isotopic into the spine of one of the components of $M - \mc{H}$, then $\gamma$ is isotopic into the spine of one of the compressionbodies of $M - H$.
\end{enumerate}
\end{lemma}

\begin{proof}
Without the loss of much generality, we may assume that $H$ is obtained from $\mc{H}$ by a single amalgamation (i.e., we can obtain $\mc{H}$ from $H$ by untelescopying in one step). The last observation is the easiest to deal with: as $H$ is obtained from $\mc{H}$ by an amalgamation, the spine of a compressionbody after amalgamation can be obtained by gluing together the 1--complexes forming the spines of two compressionbodies before amalgamation. Thus, a loop in one of those spines remains in a spine after amalgamation. 

We now proceed to prove the other observations. Suppose that $H$ decomposes $(M,T)$ into compressionbodies $A$ and $B$. Let $N_i$ for $i=1,2$ be the closure of the components of $M-\mathcal{H}^-$ on either side of $\mc{H}^-$. Let $H_i$ be the $\Gamma$-Heegaard surface of $N_i$ induced by $\mc{H}^+$. Assume that $H_1$ contains a generalized stabilization, or is perturbed, or that $T \cap N_1$ contains a removable path satisfying the hypotheses of (3).

Consider a collar $\mathcal{H}^- \times [-1,1]$ of $\mathcal{H}^-$ where $\mathcal{H}^-= \mathcal{H}^- \times \{0\}$. Recall that $H_1$ can be obtained 
from $ \mathcal{H}^- \times [0, 1]$ by attaching handles (possibly with cores running along the knot) to $ \mathcal{H}^- \times \{1\}$. Similarly, $H_2$ can
be obtained from $ \mathcal{H}^-\times [-1, 0]$ by attaching handles to $ \mathcal{H}^- \times \{-1\}$. The Heegaard surface $H$ can be obtained by extending 
the handles of $H_1$ through $ \mathcal{H}^- \times [-1, 1]$ and attaching them to $H_2$. 

{\bf Case 1a:} $H_1$ is stabilized or meridionally stabilized. Let $D$ and $E$ be the $(T, \Gamma)$-c-disks with boundary on $H_1$ defining the (meridional) stabilization. We may assume that the handles attached to $\mathcal{H}^-\times \{1\}$ include one that has $\bdd D$ as a core. The intersection of $\bdd E$ with $\mathcal{H}^- \times \{1\}$ then consists of a single arc. The possibly once punctured disk $E' = E \cup (\bdd E \times [-1, 1])$ is then a c-disk with boundary on $H$. The c-disks $E'$ and $D$ intersect exactly once and define a (meridional) stabilization of $H$. See Figure \ref{fig:stabilization}.

\begin{figure}[tbh]
\centering
\includegraphics[scale=0.4]{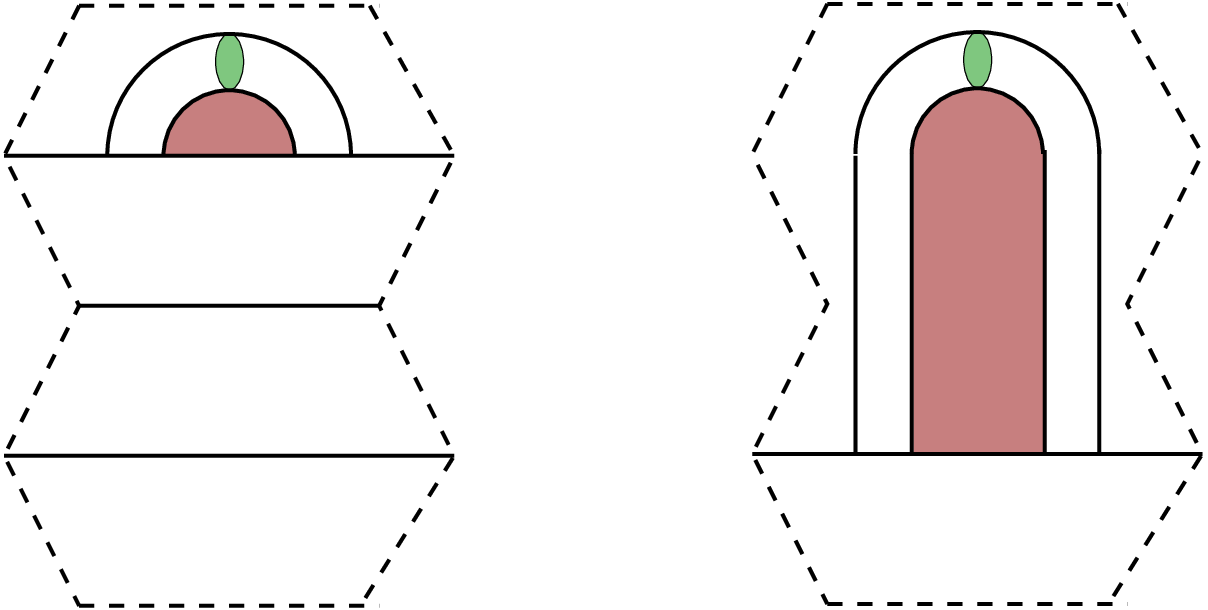}
%\scalebox{0.4}{\input{stabilization.eps_t}}
\caption{If $H_1$ is stabilized, so is $H$.} \label{fig:stabilization}
\end{figure}

{\bf Case 1b:} $H_1$ is boundary stabilized or meridionally boundary stabilized. Let $C^+_1$ and $C^+_2$ be the two $\Gamma$-compressionbodies into which $H_1$ divides $N_1$. Without loss of generality in this case $C^+_1$ contains some component $G \subset \bdd N$ and a c-disk $D$ so that removing $\eta(D)$ from $C^+_1$ decomposes it into a compressionbody $C'^+_1$ and a component $R$ homeomorphic to $G \times I$ and adding $R$ to $C^+_2$ results in a compressionbody $C'^+_2$. Let $\tilde H_1=\bdd_+ C'^+_1=\bdd_+ C'^+_2$. Amalgamating the multiple bridge splitting with thick surfaces $\tilde H_1$ and $H_1$ gives a bridge splitting for $N$ with bridge surface $\tilde H$ which can be obtained from $H$ by c-compressing along the disk $D$. Therefore $H$ is boundary stabilized or meridionally boundary stabilized. 

{\bf Case 2:} $H_1$ is perturbed. Let $D$ and $E$ be the two bridge disks for $H_1$ that intersect in one or two points and these points lie in $T$. Both disks are completely disjoint from $\mathcal{H}^- \times \{1\}$ so in particular extending the 1-handles from $H_1$ across $\mathcal{H}^- \times [-1,1]$ has no effect on these disks. Therefore $H$ is also perturbed.

{\bf Case 3:} Suppose $\zeta$ is a removable path in $N_1$ and suppose first that $\zeta$ is a cycle. Let $C_1$ and $C_2$ be the two $\Gamma$-compressionbodies $\cls(N_1 - H_1)$. As $\zeta$ is a removable cycle there exists a cancelling pair of disks $\{D_1,D_2\}$ for $\zeta$ with $D_j \subset C_j$. These disks are both completely disjoint from $\mathcal{H}^- \times \{1\}$ so in particular extending the 1-handles from $H_1$ across $\mathcal{H}^- \times [-1,1]$ has no effect on these disks.  As $\zeta$ is a removable cycle there exists a compressing disk $E$ in $C_1$, say, so that $|E \cap T| = 1$ and $|\bdd E \cap \bdd D_2| = 1$ and $E$ is otherwise disjoint from a complete collection of bridge disks for $(T\cap N_1) - H_1$ containing $D_1 \cup D_2$. We may assume that the handles attached to $\mathcal{H}^-\times \{1\}$ include one that has $\bdd E$ as a core. Therefore $E$ satisfies all the desired properties as a compressing disk for $H$.

Suppose then that $\zeta$ is a path. Again we may assume that the handles attached to $\mathcal{H}^-\times \{1\}$ include one that has $\bdd E$ as a core. The bridge disk for the component of $\zeta-H_1$ is also a bridge disk for $\zeta - H$ satisfying all of the required properties.
\end{proof}

\section{Consolidation}
We now return to the second conclusion of Lemma \ref{lem: compressible thin surface}. Scharlemann and Thompson in \cite{ST1} encounter a similar situation and are able to simply remove the two parallel surfaces from the collection of surfaces comprising the multiple Heegaard splitting. In the context of a 3-manifold without an embedded graph this results in a new multiple Heegaard splitting. When there is a graph present though, it is possible that after the removal of the two surfaces the collection of the remaining surfaces is no longer a multiple $\Gamma$-Heegaard surface, see Figure \ref{Fig: nonbridge} for an example. Thus we will need to analyze this situation very carefully. We begin by introducing some additional definitions.

\begin{figure}
\labellist \small\hair 2pt 
\pinlabel {$V$} [bl] at 26 185
\endlabellist 
\centering 
\includegraphics[scale=.4]{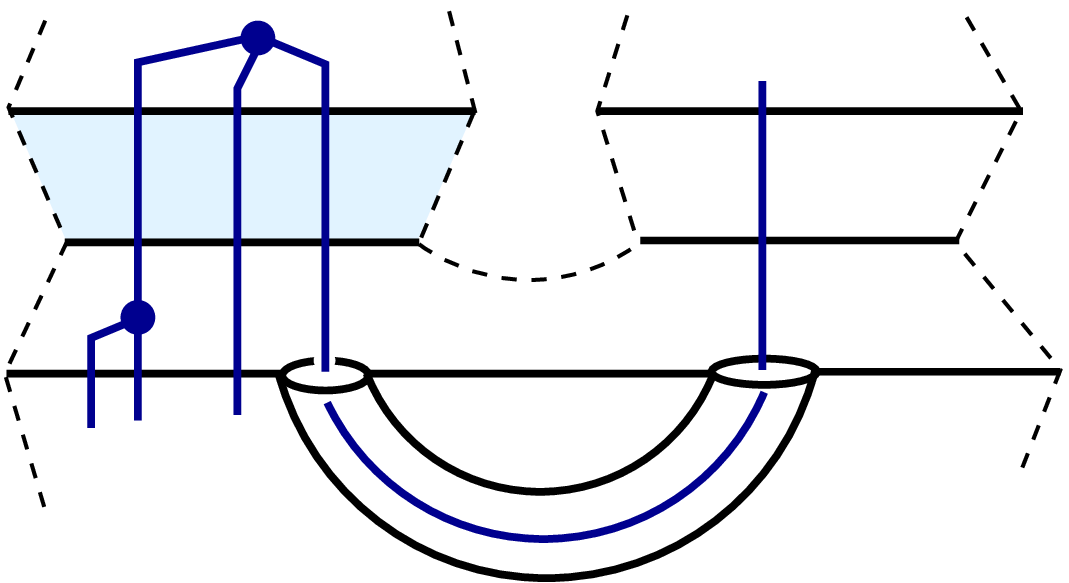}
\caption{Removing the parallel thick and thin surfaces that form the boundary of $V$ destroys bridge position. Note the presence of a pod handle and a ghost arc.}
\label{Fig: nonbridge}
\end{figure}

\subsection{Parallelism and Consolidation} Suppose that $\mc{H}$ is a multiple $\Gamma$-Heegaard splitting and that $V$ is the closure of a component of $M - \mc{H}$ such that $(V,T \cap V)$ is homeomorphic to $(F \times I, \text{points} \times I)$ for a component $F \subset \boundary V$. We allow the possibility that $T \cap V = \nil$. We say that $V$ is a {\em product region} of $\mc{H}$. A product region is a {\em parallelism} of $\mc{H}$ if it is disjoint from $\boundary M$.  If $\mc{H}$ is a multiple $\Gamma$-Heegaard splitting with a parallelism $V$, we say that $\mc{H}' = \mc{H} - (\mc{H} \cap V)$ is obtained from $\mc{H}$ by {\em consolidating} $V$. We also say that $\mc{H}'$ is obtained from $\mc{H}$ by {\em consolidation}. In what follows we develop criteria to guarantee that $\mc{H}'$ is also a multiple $\Gamma$-Heegaard splitting. 

If $V$ is a parallelism, let $V_\pm$ be the closure of the components of $M - \mc{H}$ not in $V$ that are adjacent to $\mc{H}^\pm \cap V$ respectively. Notice that each of $V_\pm$ has exactly one component. Let $T_\pm = T \cap V_\pm$.

We say that $V$ satisfies the {\em consolidation criteria} if all of the following hold:
\begin{enumerate}
\item[(C1)] No circle component of $T$ is contained entirely in $V_- \cup V \cup V_+$.
\item[(C2)] No edge of $T$ with endpoints at vertices of $T$ is contained entirely in $V_- \cup V \cup V_+$.
\item[(C3)] If $v \in T_+$ is a vertex adjacent to a handle, then there does not exist a ghost arc $\tau$ in $T_-$ connected in $V_- \cup V \cup V_+$ to $v$.
\item[(C4)] If $v$ is a vertex of $T_+$ then there is at most one edge of $T \cap (V_- \cup V \cup V_+)$ with one endpoint at $v$ and one at $\boundary_- V_+$.
\end{enumerate}

\begin{lemma}
Suppose that $\mc{H}$ is a multiple $\Gamma$-Heegaard splitting of $(M,T,\Gamma)$ and that $V$ is a parallelism of $\mc{H}$ satisfying the consolidation criteria. Let $\mc{H}'$ be obtained by consolidating $V$. Then $\mc{H}'$ is a multiple $\Gamma$-Heegaard splitting of $(M,T)$.
\end{lemma}
\begin{proof}
The multiple $\Gamma$-Heegaard splitting $\mc{H}$ defines a cobordism from $\boundary_+ V_-$ to $\boundary_- V_+$, with only 2-handles and 3-handles. We can attach all the 2-handles before all the 3-handles and so $\mc{H}'$ is a multiple Heegaard splitting of $M$. 

Let $W = V_- \cup V \cup V_+$. Suppose that $D$ is a disk in $V_+$ with $\boundary D$ the union of two arcs, one running along $T$ and the other running along $\boundary_+ V_+$. Using the product structure of $V$ and the fact that $T_-$ is trivially embedded in $V_-$, $\boundary D \cap \boundary_+ V_+$ can be extended through $V$ and $V_-$ to be a disk whose boundary is the union of an arc on $T$ and an arc on $\boundary_+ V_-$. Thus, by (C1) and (C2), $T \cap W$ is trivially embedded if and only if each vertex of $T \cap W$ is joined by at most a single edge to $\boundary_- V_+$ and if any such edge is disjoint from a complete collection of $\Gamma$-c-disks in $W$ for $\boundary_+ V_-$. Conditions (C3) and (C4) ensure that this occurs. 
\end{proof}

\begin{lemma}\label{Lem: Product Regions Post-Consol}
Let $V$ be a parallelism of $\mc{H}$ satisfying the consolidation criteria. Let $\mc{H}'$ be the multiple $\Gamma$-Heegaard splitting obtained by consolidating $V$. Then $W = V_- \cup V \cup V_+$ is a product region of $\mc{H}'$ if and only if $V_-$ and $V_+$ are product regions of $\mc{H}$.
\end{lemma}

\begin{proof}
The ``if'' direction is obvious, so we prove only the ``only if'' direction. Suppose that $W$ is a product region of $\mc{H}'$ and let $W_\pm$ be the closure of the component of $M - \mc{H}'$ adjacent to $\mc{H}'^\pm \cap V$. Notice that $W_-$ and $W_+$ are also closures of components of $M - \mc{H}$. 

Let $e$ be a component of $T \cap W$. Since $W$ is a product region, $e$ contains no vertices and joins $\boundary_- V_+$ to $\boundary_+ V_-$. Since those two surfaces are the boundary of $W$, they are parallel. Hence, there are no ghost arcs in $V_-$ or $V_+$. Thus, $e$ intersects each component of $\mc{H}^- \cap W$ exactly once. Since $V$ is a product region, $e$ intersects each component of $\mc{H} \cap W$ exactly once. Hence, $e$ is the union of vertical arcs in $V_-$, $V_+$, and $V$, and so $V_-$ and $V_+$ are product regions. 
\end{proof}

Finally we note the following:
\begin{lemma}\label{Lem: vert and ghosts}
Suppose that $\mc{H}$ is a multiple $\Gamma$-c-Heegaard surface for $(M,T,\Gamma)$. Suppose that no pod handle of $T$ is adjacent to $\mc{H}^-$ and that $\mc{K}$ is obtained by consolidating a parallelism $V$ of $\mc{H}$. If $T - \mc{K}$ has an edge joining a vertex of $T$ to $\mc{K}^-$ or an edge joining two vertices of $T$, then $T \cap V_+$ has a vertex joined in $T$ to a ghost arc of $T \cap V_-$.
\end{lemma}
\begin{proof}
Suppose that $e$ is an edge of $T - \mc{K}$ that either joins two vertices of $T$ or that joins a vertex to $\mc{K}^-$. Since no pod handle of $T$ is adjacent to $\mc{H}^-$, $e$ must intersect $V$. It does so in vertical arcs since $V$ is a parallelism. Each component of $T \cap V_-$ adjacent to $\boundary_- V$ is either a vertical arc, a ghost arc, or a vertical pod. In fact, it must be a ghost arc or vertical arc since no pod handle of $T$ is adjacent to a thin surface of $\mc{H}$. Each vertical arc in $V_-$ has an endpoint on the thick surface $\mc{K}^+ \cap V_-$, so $e \cap V_-$ must contain a ghost arc and does not contain a vertex of $T$. Hence $T \cap V_+$ has a vertex joined in $T$ to a ghost arc of $T \cap V_-$.
\end{proof}
\begin{cor}\label{Cor: vert and ghosts}
Suppose that $\mc{H}$ is a multiple $\Gamma$-c-Heegaard splitting of $(M,T,\Gamma)$ and that $V$ is a parallelism of $\mc{H}$ satisfying the following:
\begin{itemize}
\item No pod handle of $T$ is adjacent to $\mc{H}^-$
\item No vertex of $T \cap V_+$ is adjacent in $T$ to a ghost arc of $T \cap V_-$.
\end{itemize}
Then $V$ satisfies consolidation criteria (C2), (C3), and (C4).
\end{cor}
\begin{proof}
(C3) is satisfied by hypothesis. Let $\mc{K}$ be the result of consolidating $V$. If (C2) were not satisfied, then we would have an edge of $T - \mc{K}$ joining two vertices of $T$, contradicting Lemma \ref{Lem: vert and ghosts}. If (C4) were not satisfied, then we would have an edge of $T - \mc{K}$ joining a vertex to a component of $\mc{K}^-$, also contradicting Lemma \ref{Lem: vert and ghosts}. 
\end{proof}

\section{Combining untelescoping and consolidation.} We will usually use consolidation in conjunction with untelescoping. When the untelescoping operation is performed along a single pair of $\Gamma$-c-disks both thin and thick surfaces are produced. There are two possibilities for a parallelism to arise: either one of the new thick surfaces is parallel to one of the new thin surfaces or a new thick surface is parallel to a thin surfaces that existed before the untelescoping. In order to consider these possibilities in more detail we introduce the following terminology:

\begin{defin}
Suppose that $H$ is a  $\Gamma$-c-weakly reducible $\Gamma$-c-Heegaard splitting of $(M,T, \Gamma)$. A pair of disjoint $\Gamma$-c-disks $D_1$, $D_2$  on opposite sides of $H$ are called {\em maximally separated} if whenever $\{i,j\} = \{1,2\}$ the following is true:

If $E$ is a $\Gamma$-c-disk for $H$ on the same side of $H$ as $D_i$ such that $\boundary D_i$ separates $\boundary E$ from $\boundary D_j$, then $\boundary E$ is parallel in $H - T$ to $\boundary D_i$.  
\end{defin}
Notice that if $D_1$ and $D_2$ are both non-separating then they are automatically maximally separated, because the condition is vacuously true. It is clear that if $H$ is $\Gamma$-c-weakly reducible then there exist a maximally separated pair of $\Gamma$-c-disks on opposite sides of $H$.

\begin{lemma}\label{Lem: Product Details}
Let $M$ be connected. Suppose that $D_1$ and $D_2$ are a pair of maximally separated $\Gamma$-c-disks for a $\Gamma$-c-Heegaard surface $H$ for $(M,T,\Gamma)$. Assume that $\Gamma$ is disjoint from the vertices of $T-\boundary T$. Let $\mc{H}$ be the result of untelescoping $H$ using $D_1 \cup D_2$. Then the following hold:
\begin{enumerate}
\item $\mc{H}$ has at most two parallelisms and each satisfies the consolidation criteria.
\item After consolidating those parallelisms, the resulting multiple $\Gamma$-Heegaard surface has no new pod handles.
\item Suppose that $V_+$ is a component of $M - \mc{H}$ adjacent to a parallelism $V$. Then $V_+$ is one of the following:
\begin{itemize}
\item A 3--ball containing a single pod and the adjacent disk of $D_1 \cup D_2$ is not a cut disk.
\item A 3--ball containing a single arc and the adjacent disk of $D_1 \cup D_2$ is not a cut disk.
\item A product compressionbody that is either disjoint from $T$ or which intersects $T$ in vertical arcs and vertical pods.
\end{itemize}
\end{enumerate}
\end{lemma}
\begin{proof}
Untelescoping does not produce pod handles, so $\mc{H}$ has no new pod handles. Since $\Gamma$ is disjoint from the vertices of $T-\boundary T$, if either $D_1$ or $D_2$ is a cut disk, it does not intersect a pod or vertical pod of $T$. If $D_1$ and $D_2$ are non-separating then after the untelescoping there cannot be any regions of parallelism as the resulting thin surface has complexity strictly less than the complexity of either of the thick surfaces. Therefore we only need to consider the case where at least one of the disks is separating. 

Let $H_i$ be the union of the components of $\mc{H}^+$ obtained by compressing $H$ along $D_i$. Since $D_i$ is a single c-disk, each $H_i$ contains at most two components. The thin surface $\mc{H}^-$ is obtained by compressing $H$ along $D_1 \cup D_2$ and so contains at most three components. Let $C_i^\pm$ be the closures of the components of $M - \mc{H}$ adjacent to $H_i$ with $C_i^-$ adjacent to $\mc{H}^-$. Notice that $T \cap C_i^-$ consists of vertical arcs and at most one ghost arc. Since $\Gamma$ is disjoint from the vertices of $T - \boundary T$, the ghost arc does not join up to a vertex of $T - \boundary T$. See Figure \ref{Fig: Untelescoping1}. 

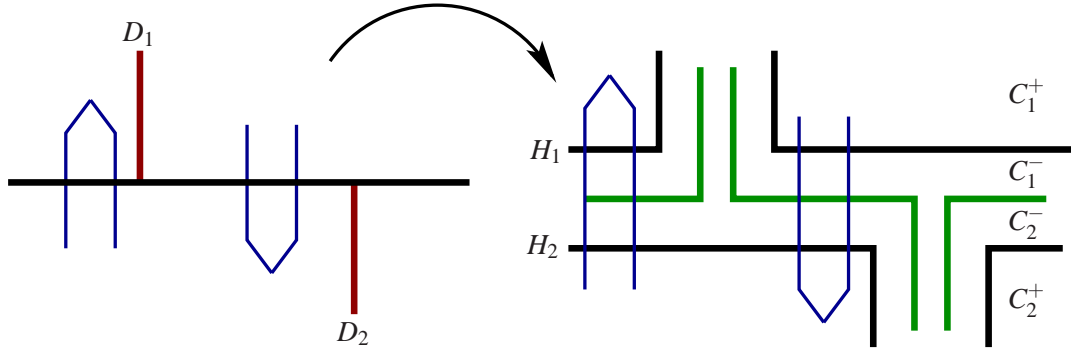
\begin{figure}
\labellist \small\hair 2pt 
\pinlabel {$D_1$} [b] at 75 165
\pinlabel {$D_2$} [t] at 192 20
\pinlabel {$H_1$} [r] at 309 110
\pinlabel {$H_2$} [r] at 307 58
\pinlabel {$C_1^-$} [b] at 560 87
\pinlabel {$C_2^-$} [t] at 560 83
\pinlabel {$C_1^+$} [b] at 560 129
\pinlabel {$C_2^+$} [b] at 560 17
\endlabellist 
\centering 
\includegraphics[scale=.7]{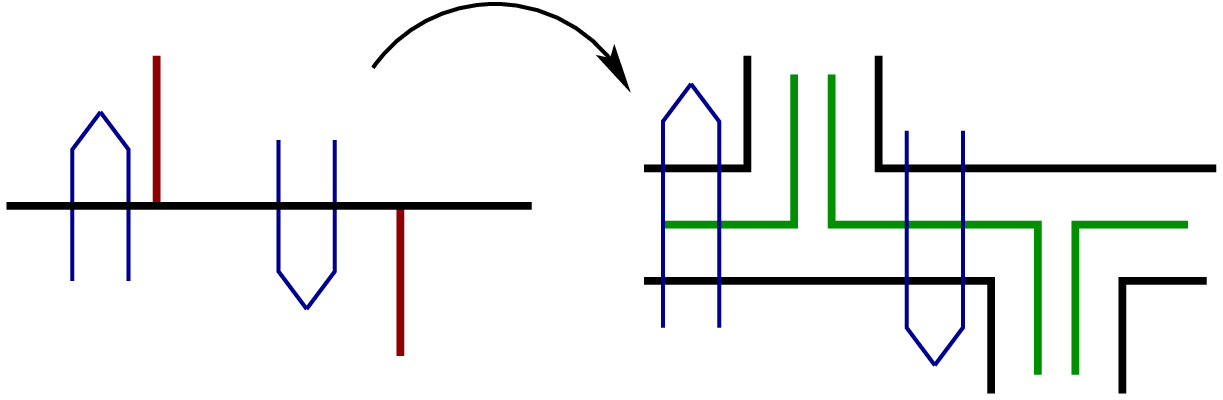}
\caption{Untelescoping a Heegaard splitting using two disks.}
\label{Fig: Untelescoping1}
\end{figure}

One component of $H_i$ is compressible in $C_i^-$ and so $C_i^-$ contains at most one product region and if it contains such a product region then it is not connected, i.e. the c-disk $D_i$ is separating. Thus $\mc{H}^-$ contains at most two parallelisms and such a parallelism is adjacent to exactly one of $D_1$ or $D_2$. The two parallelisms are shaded in Figure \ref{Fig: Parallelisms1}.

\begin{figure}
\labellist \small\hair 2pt 
\pinlabel {$V_1$} [t] at 55 91
\pinlabel {$V_2$} [b] at 214 58
\endlabellist 
\centering 
\includegraphics[scale=.8]{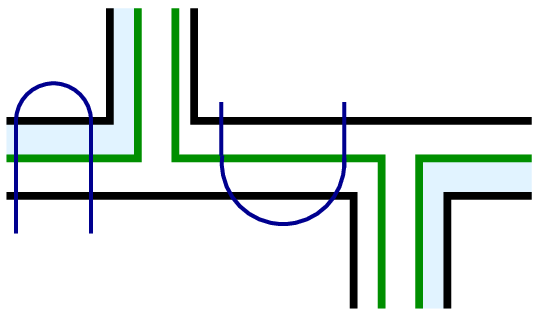}
\caption{$V_1$ and $V_2$ are the two possible parallelisms.}
\label{Fig: Parallelisms1}
\end{figure}

Since $T \cap C_i^-$ consists of vertical arcs and at most one ghost arc and since that ghost arc, if it exists, joins distinct components of $\mc{H}^-$, it is straightforward to verify that each parallelism of $\mc{H}$ satisfies the consolidation criteria. Furthermore, consolidating one of the parallelisms does not alter the other region or the components of $M - \mc{H}$ on either side of it and so it will continue to satisfy the consolidation criteria and therefore can also be consolidated. 

Let $V$ be a parallelism of $\mc{H}$. Without loss of generality, assume that $V \subset C_1^-$. Note that this implies that $D_1$ is separating and that $\boundary D_2$ is on the side of $\boundary D_1$ not adjacent to $V$. Let $H'_1 = \mc{H}^+ \cap V_+$.  If $H'_1$ had a $\Gamma$-c-disk then its boundary could be extended through $V$ to lie on $H$. The boundary of this extended disk is on the inside of $D_1$ with respect to $D_2$ and does not have boundary parallel to $\boundary D_1$. This contradicts the choice of $D_1$.

Notice that $V_+$ cannot be a 3--ball disjoint from $T$, since $D_1$ and $D_2$ were c-disks. Since $H'_1$ does not have any $\Gamma$-c-disk in $V_+$, if there exists a bridge arc of $T \cap V_+$, then $V_+$ is a 3--ball and $\tau = T \cap V_+$ is connected. Furthermore, if $D_1$ intersected $T$ (i.e. was a cut disk), then $\boundary D_1$ would bound a once-punctured disk in $H$, since $\boundary V_+$ is a 2--sphere. This contradicts the definition of cut disk, so $D_1$ is not a cut disk.

Similarly, if $T \cap V_+$ contains a bridge pod, $V_+$ must be a 3--ball and $T \cap V_+$ is connected. $D_1$ is not a cut disk since $\Gamma$ is disjoint from the vertices of $T - \boundary T$. Thus, we may assume that $T \cap V_+$ consists of vertical arcs and vertical pods. Since $H'_1$ is incompressible in $V_+$, $V_+$ must be a product compressionbody.
\end{proof}

\begin{lemma}\label{lem:removing parallelisms}
Suppose that $(M,T)$ is a (3-manifold, graph) pair. Let $\mc{K}$ be a multiple $\Gamma$-Heegaard splitting for $(M,T,\Gamma)$ such that the following hold:
\begin{enumerate}
\item $\Gamma$ is disjoint from the vertices of $T - \boundary T$ and no pod handle of $T - \mc{K}$ is adjacent to a component of $\mc{K}^-$.
\item $\mc{K}$ does not have any parallelisms.
\item No circle component of $\Gamma$ is isotopically core for $\mc{K}$.
\end{enumerate}
If some component $H$ of $\mc{K}^+$ is $\Gamma$-c-weakly reducible in $M - \mc{K}^-$ then there is a non-empty sequence of untelescopings and consolidations of $\mc{K}$ converting it into a multiple $\Gamma$-Heegaard splitting $\mc{J}$ for $(M,T)$ satisfying (1), (2), and (3).
\end{lemma}
\begin{proof}
Let $\mc{H}$ be obtained by untelescoping $H \subset \mc{K}$ using a pair of maximally separated $\Gamma$-c-disks $D_1$ and $D_2$. By Lemma \ref{Lem: Product Details}, no pod handle of $T - \mc{H}$ is adjacent to a component of $\mc{H}^-$. Since $\Gamma$ is disjoint from the vertices of $T - \boundary T$, no vertex of $\mc{H}$ is connected in $T$ to a ghost arc of $T - \mc{H}$. Consequently, all parallelisms of $\mc{H}$ satisfy consolidation criteria (C2), (C3), and (C4). Furthermore, we are guaranteed that consolidating a parallelism satisfying (C1) - (C4) produces a multiple $\Gamma$-bridge surface for $(M,T,\Gamma)$ where all parallelisms continue to satisfy (C2) - (C4).

Let $N$ be the component of $M - \mc{K}^-$ containing $H$. Any parallelism of $\mc{H}$ is contained in $N$ and by Lemma \ref{Lem: Product Regions Post-Consol}, consolidating any such parallelisms does not create additional parallelisms. Thus, it suffices to show that each parallelism of $\mc{H} \cap N$ satisfies consolidation criterion (C1) and if $V$ and $W$ are two such parallelisms then consolidating one of them does not cause the other to stop satisfying (C1).

We observe, first, that the component $F$ of $\mc{H}^- \cap N$ adjacent to both $D_1$ and $D_2$ is not parallel to either component of $\mc{H}^+ \cap N$ since the complexity of $F$ differs from the complexity of each of those surfaces. Thus, the component $F$ will never lie in the boundary of a parallelism of $\mc{H}$. Consequently, it suffices to consider only the situation in which all parallelisms of $\mc{H}$ lie on the same side of $F$ in $N$. There can be at most two such parallelisms. Let $V$ be the parallelism of $\mc{H} \cap N$ adjacent to $\mc{H}^- - \boundary N$, if such exists. Let $W$ be the parallelism of $\mc{H} \cap N$ adjacent to $\boundary N$, if such exists. See Figure \ref{Fig: Parallelisms2}.

\begin{figure}
\labellist \small\hair 2pt 
\pinlabel {$V$} [l] at 36 186
\pinlabel {$W$} [l] at 36 106
\pinlabel {$W_-$} [l] at 36 34
\pinlabel {$V_+$} [l] at 36 227
\pinlabel {$\boundary N$} [r] at 25 66
\endlabellist 
\centering 
\includegraphics[scale=.7]{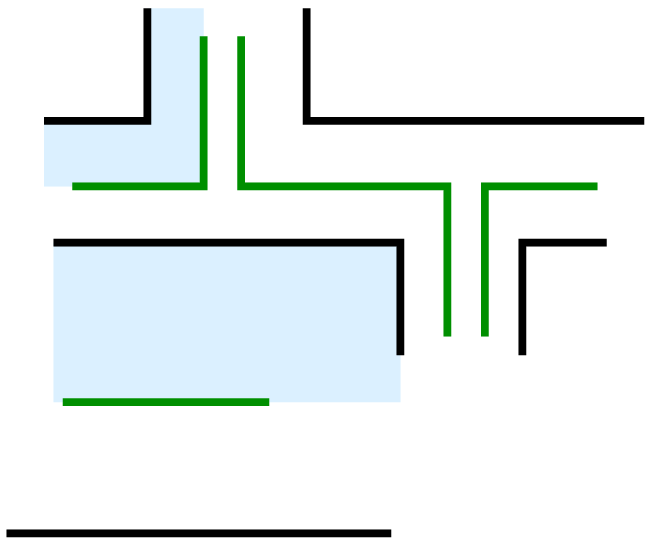}
\caption{The shaded regions denote product regions of $\mc{H}$. $V$ is a parallelism and $W$ is adjacent to $\boundary N$.}
\label{Fig: Parallelisms2}
\end{figure}
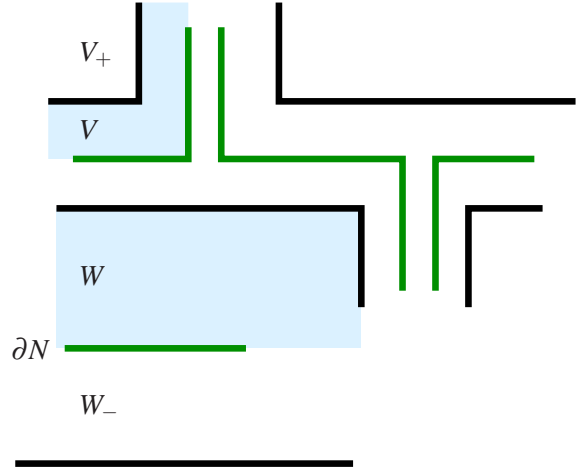

Since each compressionbody of $N - \mc{H}$ adjacent to $\mc{H}^-$ intersects $T$ only in vertical arcs and at most one ghost arc, the parallelism $V$ satisfies (C1). If $V = \nil$, then $W_+$ is a compressionbody of $N - \mc{H}$ adjacent to $\mc{H}^-$. The ghost arc in such a compressionbody, if it exists, is disjoint from $\boundary W$. Hence, if $V = \nil$ the parallelism $W$ satisfies the consolidation criteria. Suppose, therefore, that $V \neq \nil$ and $W \neq \nil$. Let $\mc{H}'$ be the result of consolidating $V$ in $\mc{H}$.

Assume that $W$ does not satisfy (C1) and let $\tau \subset (W_- \cup W \cup W_+)$ be a circle component of $T$. Since $\tau \cap W$ consists of vertical arcs, $\tau \cap W_+$ must have a bridge arc. By Lemma \ref{Lem: Product Details}, such a bridge arc intersects $V_+$. By that lemma, $V_+$ is a 3--ball, the component of $D_1 \cup D_2$ adjacent to it is a not a cut-disk and $T \cap V_+$ consists only of a single bridge arc. Thus, $\tau \cap W \cup W_+$ consists of a single arc. Hence, $\tau \cap W_-$ consists of a single arc with both endpoints on $\boundary N \cap \mc{K}^-$. Such an arc is a ghost arc in $W_-$ and so must belong to $\Gamma$. Hence $\tau \subset \Gamma$. See Figure \ref{Fig: Parallelisms3}. Since $\tau$ is the union of a ghost arc in $\tau_-$ with an arc in $W \cup W_+$ that is isotopic into $\boundary_- W_-$. Thus, $\tau$ can be isotoped to lie in the spine of the compressionbody $W_- $. Hence, by Lemma \ref{lem:stabilization}, $\tau$ is isotopically core with respect to $\mc{K}$, a contradiction. 

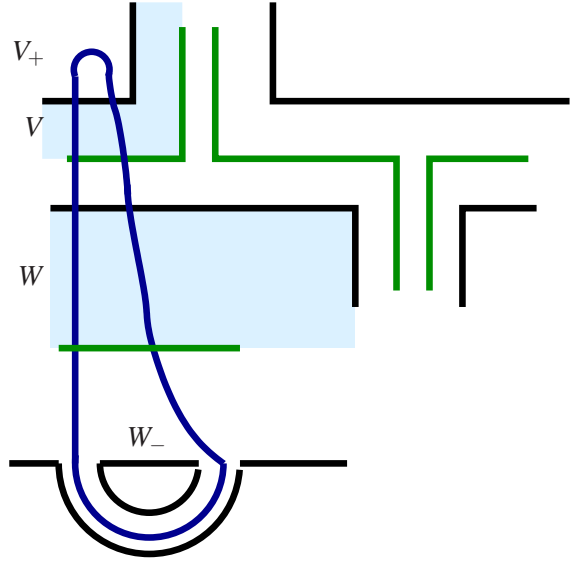
\begin{figure}
\labellist \small\hair 2pt 
\pinlabel {$V$} [r] at 26 236
\pinlabel {$W$} [r] at 26 154
\pinlabel {$W_-$} [b] at 79 55
\pinlabel {$V_+$} [r] at 26 276
\endlabellist 
\centering 
\includegraphics[scale=.7]{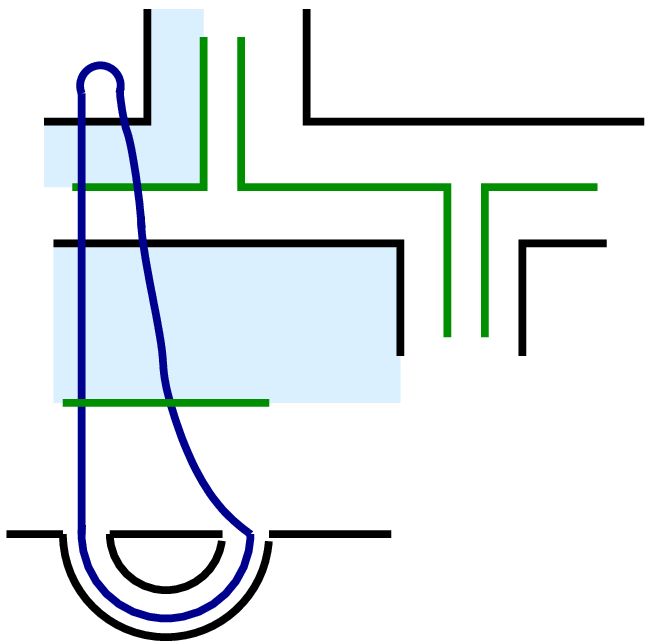}
\caption{The circle $\tau$.}
\label{Fig: Parallelisms3}
\end{figure}
\end{proof}

\section{Untelescoping and Essential Surfaces}
In this section, we put together the previous results on untelescoping and consolidation and prove the main theorem. We begin by naming the main properties we seek.

\begin{defin}
Suppose that $\mc{K}$ and $\mc{H}$ are multiple $\Gamma$-c-Heegaard splitting for $(M,T,\Gamma)$. If $\mc{H}$ is obtained from $\mc{K}$ by a sequence of untelescopings and consolidations, then we say that $\mc{H}$ is obtained by {\em thinning} $\mc{K}$. If $\mc{H}$ has the property that each component of $\mc{H}^+$ is $\Gamma$-c-irreducible in $M - \mc{H}^-$, each component of $\mc{H}^-$ is $\Gamma$-c-incompressible, and no component of $\mc{H}^-$ bounds a parallelism with a component of $\mc{H}^+$, then we say that $\mc{H}$ is {\em slim}. If $\mc{H}$ is slim and if each component of $\mc{H}^-$ is $T$-essential, then $\mc{H}$ is {\em locally thin}. 
\end{defin}

We begin with Theorem \ref{thm: incompressible thin} which shows that we can obtain slim multiple Heegaard splittings. Theorem \ref{thm:main} will strengthen this to produce locally thin multiple Heegaard splittings.

\begin{thm}\label{thm: incompressible thin}
Let $M$ be a compact, orientable $3$-manifold containing a properly embedded graph $T$ and let $\Gamma$ be a subgraph of $T$ disjoint from the vertices of $T - \boundary T$. Furthermore assume that $M-T$ is irreducible and that no sphere in $M$ intersects $T$ exactly once. Suppose $\mc{K}$ is a $\Gamma$-bridge surface for $(M, T,\Gamma)$ such that the following hold:
\begin{itemize}
\item $\mc{K}$ has no parallelisms
\item no pod handle of $T - \mc{K}$ is adjacent to $\mc{K}^-$
\item no component of $\Gamma$ is isotopically core with respect to $\mc{K}$. 
\end{itemize}
Then there is a slim multiple $\Gamma$-bridge splitting $\mc{H}$ for $(M,T, \Gamma)$ obtained by thinning $\mc{K}$.
\end{thm}

\begin{proof}
We define a sequence $(\mc{H}_i)$ of multiple $\Gamma$-c-Heegaard splittings such that each $\mc{H}_i$ satisfies the hypotheses of the theorem and so that for all $i$, $\mc{H}_{i+1}$ is obtained from $\mc{H}_i$ by a sequence of untelescopings using $\Gamma$-c-disks and consolidations. Assume that such an $\mc{H}_i$ does not satisfy the conclusions of the theorem. Since $\mc{H}_i$ has no parallelisms, if a component of $\mc{H}_i^-$ is $\Gamma$-c-compressible, then by Lemma \ref{lem: compressible thin surface}, there is a component of $\mc{H}_i^+$ that is $\Gamma$-c-weakly reducible in $M - \mc{H}_i^-$. Hence, if $\mc{H}_i$ does not satisfy the conclusions of the theorem, some component of $\mc{H}_i^+$ is $\Gamma$-c-weakly reducible in $M - \mc{H}_i^-$.

We define $(\mc{H}_i)$ inductively as follows. Let $\mc{H}_1 = \mc{K}$. Assume that $\mc{H}_i$ has been defined. If $\mc{H}_i$ satisfies the conclusions of the theorem, let $\mc{H}_{i+1} = \mc{H}_i$. If $\mc{H}_{i}$ does not satisfy the conclusions of the theorem, let $\mc{H}_{i+1}$ be the multiple $\Gamma$-c-Heegaard splitting of $(M,T,\Gamma)$ obtained by applying Lemma \ref{lem:removing parallelisms} to $\mc{H}_i$. By that lemma, $\mc{H}_{i+1}$ is obtained from $\mc{H}_i$ by a sequence of untelescopings using $\Gamma$-c-disks and consolidations and it satisfies the hypotheses of the theorem. Hence, we obtain a sequence $(\mc{H}_i)$ as desired.

Since both untelescoping and consolidation strictly reduce complexity, this sequence is eventually constant at a $\Gamma$-multiple Heegaard splitting $\mc{H}$ of $(M,T, \Gamma)$ satisfying the conclusions of the theorem.
\end{proof}

We can now prove our main theorem.

\begin{thm} \label{thm:main}
Let $M$ be a compact, orientable $3$-manifold containing a properly embedded graph $T$ and let $\Gamma$ be a subgraph of $T$ disjoint from the vertices of $T - \boundary T$. Furthermore assume that $M-T$ is irreducible and that no sphere in $M$ intersects $T$ exactly once. Suppose $\mc{K}$ is a multiple $\Gamma$-bridge surface for $(M, T,\Gamma)$ such 
\begin{itemize}
\item $\mc{K}$ has no parallelisms
\item no pod handle of $T - \mc{K}$ is adjacent to $\mc{K}^-$
\item no circle or edge component of $\Gamma$ is isotopically core with respect to $\mc{K}$. 
\end{itemize}
Then there is a multiple $\Gamma$-Heegaard splitting $\mc{H}$ for $(M,T,\Gamma)$ obtained by thinning $\mc{K}$ such that one of the following holds:
\begin{itemize}
\item $\mc{H}$ is locally thin, or
\item some component of $\mc{H}^+$ contains a generalized stabilization in $M - \mc{H}^-$, or
\item some component of $\mc{H}^+$ is perturbed in $M - \mc{H}^-$, or
\item $T$ has a removable path.
\end{itemize}
\end{thm}

\begin{proof}
Let $\mc{H}$ be the multiple $\Gamma$-Heegaard splitting of $(M,T, \Gamma)$ obtained by applying Theorem \ref{thm: incompressible thin} to $\mc{K}$. By that theorem, each component of $\mc{H}^+$ is $\Gamma$-c-weakly reducible in $M - \mc{H}^-$, each component of $\mc{H}^-$ is $\Gamma$-c-incompressible, and no component of $\mc{H}^-$ bounds a parallelisms with a component of $\mc{H}^+$. 

Suppose that a component $F$ of $\mathcal{H}^-$ is $T$-parallel. Let $C_F$ be the compressionbody so that $F = \bdd_+ C_F$ is parallel to the boundary of a regular neighborhood of some components of $\bdd M$ together with some subset of $T$. By Lemma \ref{lem:essential} we may assume that $F$ is innermost, i.e., $C_F$ does not contain any other thin surfaces. Let $H_{C_F}$ be the $\Gamma$-Heegaard splitting for $C_F$ given by the unique thick surface of $\mathcal{H}$ contained in $C_F$. 

\textbf{Case 1:} $T \cap C_F = \nil$.

Recall that $F$ is not parallel to $H_{C_F}$. If $\boundary_- C_F = \nil$, then by \cite{W}, $H$ is stabilized. If $\boundary_- C_F \neq \nil$, by \cite{ST2} $H_{C_F}$ is stabilized or boundary stabilized along $\bdd_-C_F$. See also \cite{MS}. 

\textbf{Case 2:} $H_{C_F}$ is a Heegaard splitting.

Since no thin surface is parallel to a thick surface, by Theorem \ref{thm:HSofCompBodies2} one of the following occurs:
\begin{itemize}
\item $H_{C_F}$ has a generalized stabilization and if this is a boundary stabilization, it is along a component of $\bdd M$; 
\item $H_{C_F}$ is perturbed; or
\item $T \cap C_F$ has a removable path disjoint from $F$.
\end{itemize}
 All of these conclusions are possible conclusions in the theorem at hand.

\textbf{Case 3:} $H_{C_F}$ is a $\Gamma$-Heegaard splitting but not a Heegaard splitting.

Let $A$ and $B$ be the two compressionbodies into which $H_{C_F}$ divides $C_F$. Since $H_{C_F}$ is a $\Gamma$-Heegaard splitting but not a Heegaard splitting, there exists an edge $e \subset \Gamma$ in either $A$ or $B$ which is disjoint from $H_{C_F}$ and which has both endpoints on $\boundary C_F$.

\textbf{Case 3a:} $\boundary e \subset F$.

Since $F$ is parallel to $\boundary M \cup T$ and since $e \subset T$ is an edge with both endpoints on $F$, $T = e$ and $F = S^2$. Then by \cite[Lemma 2.1]{HS1} and \cite[Theorem 1.1]{HS3}, either $H_{C_F}$ is stabilized, meridionally stabilized, or perturbed. See Case 2 of the proof of \cite[Lemma 5.2]{T1} for details. 

\textbf{Case 3b:} $\boundary e \subset \boundary M$

In this case $e$ is disjoint not only from $H_{C_F}$ but also from $\mc{H}$. Then $e$ is an edge of $\Gamma$ with both endpoints on $\boundary M$ which is disjoint from $\mc{H}$. Hence $e$ is isotopically core with respect to $\mc{K}$, a contradiction.

\textbf{Case 3c:} One endpoint of $e$ is on $F$ and one endpoint of $e$ is on $\boundary M$.

We may assume that the hypotheses of cases (3a) and (3b) do not apply. Let $e_1, \hdots, e_n$ be the union of edges of $T \cap C_F$ with one endpoint on $F$, one endpoint on $\boundary M$, and which are disjoint from $H_{C_F}$. Perform a slight isotopy of each of them to convert $T$ into a graph $T'$ and each edge $e_i$ into $e'_i$ so that each $e'_i$ is a removable edge of $T'$ as in Lemma \ref{lem: creating removable edges}. Let $H'_{C_F}$ be the new $\Gamma$-Heegaard surface and notice that $H'_{C_F}$ is, in fact, a Heegaard surface for $C_F$. Since $F$ is not parallel to $H_{C_F}$, by Theorem \ref{thm:HSofCompBodies2} one of the following occurs:
\begin{itemize}
\item $H'_{C_F}$ has a generalized stabilization and if this is a boundary stabilization, it is along a component of $\bdd M$;  
\item $H'_{C_F}$ is perturbed; or
\item $T \cap C_F$ has a removable path with both endpoints in $\boundary M$.
\end{itemize}
Notice that if the last option occurs the removable path is not equal to any of the $e'_i$ since each of those edges has one endpoint on $F$. By Lemma \ref{lem: creating removable edges} one of the following occurs:
\begin{itemize}
\item $H_{C_F}$ has a generalized stabilization and if this is a boundary stabilization, it is along a component of $\bdd M$; 
\item $H_{C_F}$ is perturbed; or
\item $T \cap C_F$ has a removable path disjoint from $F$.
\end{itemize}
All of these are conclusions for the theorem at hand. \end{proof}

Combining the above result with Lemma \ref{lem:stabilization} gives us an almost immediate corollary for bridge surfaces:
\begin{cor} \label{cor:main}
Let $M$ be a compact, orientable $3$-manifold containing a properly embedded graph $T$ and let $\Gamma$ be a subgraph of $T$ disjoint from the vertices of $T - \boundary T$. Furthermore assume that $M-T$ is irreducible and that no sphere in $M$ intersects $T$ exactly once. Suppose $K$ is a bridge surface for $(M, T,\Gamma)$ such no circle or edge component of $\Gamma$ is isotopically core with respect to $K$.  Then one of the following holds:
\begin{itemize}
\item there is a slim multiple $\Gamma$-Heegaard splitting $\mc{H}$ for $(M,T, \Gamma)$ obtained by thinning $K$. Furthermore, if $K$ is not thin, then $\mc{H}^- \neq \nil$.
\item $K$ contains a generalized stabilization,
\item $K$ is perturbed, or
\item $K$ has a removable path.
\end{itemize}
\end{cor}
\begin{proof}

By Theorem \ref{thm:main}, either $T$ has a removable path or there is a multiple $\Gamma$-Heegaard splitting $\mc{H}$ for $(M,T,\Gamma)$ obtained by thinning $K$ such that one of the following holds:
\begin{itemize}
\item $\mc{H}$ is locally thin, or
\item some component of $\mc{H}^+$ contains a generalized stabilization in $M - \mc{H}^-$, or
\item some component of $\mc{H}^+$ is perturbed in $M - \mc{H}^-$.
\end{itemize}
Note that $K$ can be recovered from $\mc{H}$ by a sequence of amalgamations so by Lemma \ref{lem:stabilization} if $\mc{H}^+$ contains a generalized stabilization or perturbation in $M - \mc{H}^-$ so does $K$. Thus the only statement we need to verify is that if $K$ is not thin, then $\mc{H}^- \neq \nil$. To see this, notice that when a bridge surface is untelescoped using a pair of maximally separated disks, the thin surface adjacent to both disks is never adjacent to one of the parallelisms that is possibly created by the untelescoping. Thus it cannot be consolidated until it is adjacent to a thick surface that is untelescoped to create a thick surface parallel to it. Thus, throughout the thinning process there is always a thin surface, although the thin surface may change as the sequence progresses.
\end{proof}

\end{document}